\documentclass[reqno]{amsproc}
\usepackage{amssymb,url}
 \usepackage[backref]{hyperref}

\usepackage[all, knot]{xy}
\xyoption{arc}

\makeatletter
 \def\activeat#1{\csname @#1\endcsname}
 \def\def@#1{\expandafter\def\csname @#1\endcsname}
 {\catcode`\@=\active \gdef@{\activeat}}

\let\ssize\scriptstyle
\newdimen\ex@	\ex@.2326ex

 \def\requalfill{\cleaders\hbox{$\mkern-2mu\mathord=\mkern-2mu$}\hfill
  \mkern-6mu\mathord=$}
 \def\eqfill{$\m@th\mathord=\mkern-6mu\requalfill}
 \def\deffill{\hbox{$:=$}$\m@th\mkern-6mu\requalfill}
 \def\fiberbox{\hbox{$\vcenter{\hrule\hbox{\vrule\kern1ex
     \vbox{\kern1.2ex}\vrule}\hrule}$}}

 \newdimen\arrwd 
  \newdimen\minCDarrwd \minCDarrwd=2.5pc
   \setbox\z@\hbox{$\rightarrow\,$} \minCDarrwd=\wd\z@
 	
 \def\findarrwd#1#2#3{\arrwd=#3%
  \setbox\z@\hbox{$\ssize\;{#1}\;\;$}%
  \setbox\@ne\hbox{$\ssize\;{#2}\;\;$}%
  \ifdim\wd\z@>\arrwd \arrwd=\wd\z@\fi
  \ifdim\wd\@ne>\arrwd \arrwd=\wd\@ne\fi}
 \newdimen\arrowsp\arrowsp=0.375em  	
 \def\findCDarrwd#1#2{\findarrwd{#1}{#2}{\minCDarrwd}
    \advance\arrwd by 2\arrowsp}
 \newdimen\minarrwd 
 \setbox\z@\hbox{$\longrightarrow$} \minarrwd=\wd\z@

 \def\harrow#1#2#3#4{{\minarrwd=#1\minarrwd
   \findarrwd{#2}{#3}{\minarrwd}\kern\arrowsp
    \mathrel{\mathop{\hbox to\arrwd{#4}}\limits^{#2}_{#3}}\kern\arrowsp}}

 \def@]#1>#2>#3>{\harrow{#1}{#2}{#3}\rightarrowfill}
 \def@>#1>#2>{\harrow1{#1}{#2}\rightarrowfill}
 \def@<#1<#2<{\harrow1{#1}{#2}\leftarrowfill}
 \def@={\harrow1{}{}\eqfill}
 \def@:#1={\harrow1{}{}\deffill}
 \def@ N#1N#2N{\vCDarrow{#1}{#2}\UpDownarrow}
 \def\UpDownarrow{\uparrow\,\Big\downarrow}

\def\hookrightarrowfill{\hbox{$\lhook\joinrel$}\rightarrowfill}
\def@{hk>}#1>#2>{\harrow1{#1}{#2}\hookrightarrowfill}
\def@ c>#1>#2>{\harrow1{#1}{#2}\hookrightarrowfill}
\def\hookleftarrowfill{\leftarrowfill\hbox{$\joinrel\rhook$}}
\def@{hk<}#1<#2<{\harrow1{#1}{#2}\hookleftarrowfill}

 \def@.{\ifodd\row\relax\harrow1{}{}\hfill
   \else\vCDarrow{}{}.\fi}
 \def@|{\vCDarrow{}{}\Vert}
 \def@ V#1V#2V{\vCDarrow{#1}{#2}\downarrow}
 \def@ A#1A#2A{\vCDarrow{#1}{#2}\uparrow}
 \def@(#1){\arrwd=\csname col\the\col\endcsname\relax
   \hbox to 0pt{\hbox to \arrwd{\hss$\vcenter{\hbox{$#1$}}$\hss}\hss}}

 \def\squash#1{\setbox\z@=\hbox{$#1$}\finsm@@sh}
\def\finsm@@sh{\ifnum\row>1\ht\z@\z@\fi \dp\z@\z@ \box\z@}

 \newcount\row \newcount\col \newcount\numcol \newcount\arrspan
 \newdimen\vrtxhalfwd  \newbox\tempbox

 \def\innernewdimen{\alloc@1\dimen\dimendef\insc@unt}
 \def\measureinit{\col=1\vrtxhalfwd=0pt\arrspan=1\arrwd=0pt 
   \setbox\tempbox=\hbox\bgroup$}
 \def\setinit{\col=1\hbox\bgroup$\ifodd\row
   \kern\csname col1\endcsname
   \kern-\csname row\the\row col1\endcsname\fi}
 \def\findvrtxhalfsum{$\egroup
  \expandafter\innernewdimen\csname row\the\row col\the\col\endcsname
  \global\csname row\the\row col\the\col\endcsname=\vrtxhalfwd
  \vrtxhalfwd=0.5\wd\tempbox
  \global\advance\csname row\the\row col\the\col\endcsname by \vrtxhalfwd 
  \advance\arrwd by \csname row\the\row col\the\col\endcsname
  \divide\arrwd by \arrspan
  \loop\ifnum\col>\numcol \numcol=\col%
     \expandafter\innernewdimen \csname col\the\col\endcsname
     \global\csname col\the\col\endcsname=\arrwd
   \else \ifdim\arrwd >\csname col\the\col\endcsname
      \global\csname col\the\col\endcsname=\arrwd\fi\fi
   \advance\arrspan by -1 %
   \ifnum\arrspan>0 \repeat}
 \def\setCDarrow#1#2#3#4{\advance\col by 1 \arrspan=#1 
    \arrwd= -\csname row\the\row col\the\col\endcsname\relax
    \loop\advance\arrwd by \csname col\the\col\endcsname
     \ifnum\arrspan>1 \advance\col by 1 \advance\arrspan by -1%
     \repeat
    \squash{\mathop{
     \hbox to\arrwd{\kern\arrowsp#4\kern\arrowsp}}\limits^{#2}_{#3}}}
 \def\measureCDarrow#1#2#3#4{\findvrtxhalfsum\advance\col by 1%
   \arrspan=#1\findCDarrwd{#2}{#3}%
    \setbox\tempbox=\hbox\bgroup$}
 \def\vCDarrow#1#2#3{\kern\csname col\the\col\endcsname
    \hbox to 0pt{\hss$\vcenter{\llap{$\ssize#1$}}%
     \Big#3\vcenter{\rlap{$\ssize#2$}}$\hss}\advance\col by 1}

 \def\setCD{\def\harrow{\setCDarrow}%
  \def\\{$\egroup\advance\row by 1\setinit}
  \m@th\lineskip3\ex@\lineskiplimit3\ex@ \row=1\setinit}
 \def\endsetCD{$\egroup}
 \def\dr@p#1\\{\findvrtxhalfsum\advance\row by 2 \measureinit}
 \def\measure{\bgroup
  \def\harrow{\measureCDarrow}%
  \def\\##1{\ifx##1\endmeasure\endmeasure\else\expandafter\dr@p\fi}%
  \row=1\numcol=0\measureinit}
 \def\endmeasure{\findvrtxhalfsum\egroup}

 \newcount\savedcount	
 \def\LCD#1\end{\savedcount=\count11
   \measure#1\endmeasure
   \vcenter{\setCD#1\endsetCD\kern\medskipamount}%
   \global\count11=\savedcount\end}

 \newenvironment{CD}{\let\at=@\catcode`\@=\active\LCD}{\catcode`\@=12\relax}
\makeatother

 \usepackage{fancyhdr} 

\def\thetitle{Enriques diagrams, arbitrarily near points,\\
 and Hilbert schemes}
\newcommand{\emdash}{\unskip\penalty10000\thinspace
        ---\penalty-500\thinspace\ignorespaces}

\def\Wt(#1){m_{#1}} \let\wh=\widehat
\let\wt=\widetilde \let\td=\tilde
\def\UThin{\penalty\@M \thinspace\ignorespaces}
\def\(#1){{\let~=\UThin\rm(#1)}}
\def\tsum{\textstyle\sum}

 \DeclareMathOperator{\Aut}{Aut}

\DeclareMathOperator{\Hilb}{Hilb}

 \DeclareMathOperator{\Hom}{Hom}
\DeclareMathOperator{\Ker}{Ker}

\DeclareMathOperator{\Proj}{Proj}

\DeclareMathOperator{\Spec}{Spec}
\DeclareMathOperator{\Sym}{Sym}

\DeclareMathOperator{\cod}{cod}

\DeclareMathOperator{\type}{type}

\let\?=\overline
\let\:=\colon
\let\bb=\mathbb
\let\bu=\bullet
\let\I=\mathbf
\let\into=\hookrightarrow
\let\mc=\mathcal
\let\onto=\twoheadrightarrow
\let\ox=\otimes
\let\tu=\textup
\let\vf=\varphi
\let\x=\times
\let\xto=\xrightarrow

\def\risom{\buildrel\sim\over{\smashedlongrightarrow}}
 \def\smashedlongrightarrow{\setbox0=\hbox{$\longrightarrow$}\ht0=1.25pt\box0}
\newcommand{\uH}{\tu H}
\newcommand{\uZ}{\tu Z}
\newcommand{\uB}{\tu B}
\newcommand{\AAA}{{\mathbb A}}
\newcommand{\CI}{{\mathcal I}}
\newcommand{\CJ}{{\mathcal J}}

\newcommand{\CO}{{\mathcal O}}

\theoremstyle{plain}
 \newtheorem{thm}{Theorem}[section]
 \newtheorem{cor}[thm]{Corollary}
 \newtheorem{gss}[thm]{Guess}
 \newtheorem{lem}[thm]{Lemma}
 \newtheorem{prp}[thm]{Proposition}

\theoremstyle{definition}
 \newtheorem{dfn}[thm]{Definition}
 \newtheorem{eg}[thm]{Example}

  \newtheorem{rmk}[thm]{Remark}
 \newtheorem{sbs}[thm]{}

\makeatletter
 \@addtoreset{equation}{thm}
\makeatother

\def\mylistparam
 {\setbox0=\hbox{\rm(1)\kern0.5em}
  \setlength{\labelwidth}{\wd0}
  \setlength{\topsep}{\smallskipamount}%
  \setlength{\itemsep}{0pt}\setlength{\parsep}{0pt}%
  \setlength{\itemindent}{0pt}%
  \setlength{\leftmargin}{\parindent}%
  \addtolength{\leftmargin}{\wd0}%
 }%
\renewenvironment{enumerate}%
 {\begin{list}{\rm(\arabic{enumi})}%
  {\usecounter{enumi}\mylistparam}%
  }%
 {\end{list}}

\begin{document}

\title\thetitle

\author[S. Kleiman]{Steven KLEIMAN}
 \address
 {Department of Mathematics, MIT \\
 77 Massachusetts Avenue\\
 Cambridge, MA 02139, USA}
 \email{kleiman@math.mit.edu}

\author[R. Piene]{Ragni PIENE}
\address{CMA and Department of Mathematics\\
University of Oslo\\ 
PO Box 1053, Blindern\\
NO-0316 Oslo, Norway}
\email{ragnip@math.uio.no}

\author{\\with  Appendix B by Ilya TYOMKIN}
 \address
 {Department of Mathematics\\
 Ben Gurion University of the Negev\\
 P.O.B. 653, Be'er Sheva 84105, Israel}
 \email{tyomkin@cs.bgu.ac.il}

\date{\today}

\subjclass[2000]{Primary 14N10; Secondary 14C20, 14H40, 14K05}

\keywords{Enriques diagrams, arbitrarily near points, Hilbert schemes,
  property of exchange.}

\begin{abstract}

Given a smooth family $F/Y$ of geometrically irreducible surfaces, we
study sequences of {\it arbitrarily near} $T$-points of $F/Y$; they
generalize the traditional sequences of infinitely near points of a
single smooth surface.  We distinguish a special sort of these new
sequences, the {\it strict} sequences.  To each strict sequence, we
associate an ordered unweighted Enriques diagram.  We prove that the
various sequences with a fixed diagram form a functor, and we represent
it by a smooth $Y$-scheme.

We equip this $Y$-scheme with a free action of the automorphism group of
the diagram.  We equip the diagram with weights, take the subgroup of
those automorphisms preserving the weights, and form the corresponding
quotient scheme.  Our main theorem constructs a canonical universally
injective map from this quotient scheme to the Hilbert scheme of $F/Y$;
further, this map is an embedding in characteristic 0.  However, in
every positive characteristic, we give an example,  in Appendix
B, where the map is
purely inseparable.
 \end{abstract}

\maketitle

\section{Introduction}\label{sc:intro}
\parskip=0pt plus 2pt
Recently, there has been much renewed interest in an old and timeless
problem: enumerating the $r$-nodal curves in a linear system
on a smooth projective surface.  Notably, Tzeng \cite{Tzeng}
and Kool, Shende and Thomas \cite{KST}, in very different ways, proved
G\"ottsche's conjecture \cite{Gtt98}, which was motivated by the
Yau--Zaslow formula \cite{YZ} and describes the shape of a
corresponding generating series in $r$.  However, G\"ottsche's
expression involves two power series, which remain unknown in general.

On the other hand, Vainsencher \cite{V95} found enumerating polynomials
for $r\le 7$; they are explicit and general, although he
\cite[Sec.\,7]{V95} was unsure of the case  $r=7$.  In \cite{KP99}
and \cite{K--P}, the present authors refined and extended Vainsencher's
work, by settling the case $r=7$, handling $r=8$, producing more compact
formulas, and establishing validity under more extensive conditions.
Further, this enumeration, unlike G\"ottsche's, applies to
nonconstant families of surfaces; notably, see  \cite[pp.\,513--514]{V95}
and \cite[pp.\,80--83]{K--P}, it proves 17,601,000 is the number of
irreducible 6-nodal quintic plane curves on a general quintic 3-fold in
4-space, contrary to the predictions of Clemens' conjecture and of mirror
symmetry.

To establish a range of validity for these enumerative formulas, it is
necessary to analyze various generalized Severi varieties, namely, the
loci of curves of given equisingularity type in the system.  G\"ottsche
\cite[Prp.\,5.2]{Gtt98} treated nodes in an ad hoc fashion; Tzeng relies
on his analysis, whereas Kool, Shende and Thomas \cite[Prp.\,2.1]{KST}
improved it, thus extending the range of validity.

On the other hand, Vainsencher's approach, as pursued by the present
authors, relies on a more extensive and more systematic analysis.  It is
based on Enriques diagrams.  They are directed graphs, similar to
resolution graphs, that represent the equisingularity types of the
curves.  Equivalently, see \cite[\S\,3]{KP99} and the references 
there, they represent the types of the complete ideals, the ideals
formed by the equations of the curves with singularities of the same
type or worse at given points.

Specifically, in the authors' paper~\cite{KP99}, Proposition (3.6) on
p.\,225 concerns the locus $H({\bf D})$ that sits in the Hilbert scheme
of a smooth irreducible complex surface and parametrizes the
complete ideals $\mc I$ with a given minimal Enriques diagram ${\bf D}$.
The proposition asserts that $H({\bf D})$ is smooth and equidimensional.

The proposition was justified intuitively, then given an ad hoc proof in
\cite{KP99}.  The intuitive justification was not developed into a
formal proof, which is surprisingly long and complicated.
However, the proof yields more: it shows  $H({\bf D})$ is
irreducible; it works for nonminimal ${\bf D}$; and it works for
\emph{families} of surfaces.  Further, it works to a great extent when
the characteristic is positive or mixed, but then it only shows 
$H({\bf D})$ has a finite and universally injective covering by a smooth
cover; this covering need not be birational, as  examples in Appendix
B show.

It is naive to form $H({\bf D})$ as a locus with reduced scheme
structure.  It is more natural to consider the functor of sequences of
arbitrarily near points corresponding to ${\bf D}$.  This functor is
representable by a smooth irreducible scheme, and it admits a natural map
into the Hilbert scheme, whose image is $H({\bf D})$.  This map is
finite and universally injective, so an embedding in characteristic zero,
but it may be totally ramified in positive characteristic as the
examples show.

Originally, the authors planned to develop this discussion in a paper
that also dealt with other loose ends, notably, the details of the
enumeration of curves with eight nodes.  However, there is so much
material involved that it makes more sense to divide it up.  Thus the
discussion of $H({\bf D})$ alone is developed in the present paper; the
result itself is asserted in Corollary~\ref{corHsbs}.  Here, in more
detail, is a description of this paper's contents.

Fix a smooth family of geometrically irreducible surfaces $F/Y$ and an
integer $n\ge0$.  Given a $Y$-scheme $T$, by a {\it sequence of
  arbitrarily near $T$-points of\/} $F/Y$, we mean an $(n+1)$-tuple
$(t_0,\dots,t_n)$ where $t_0$ is a $T$-point of $F^{(0)}_T:=F\x_YT$ and
where $t_i$, for $i\ge1$, is a $T$-point of the blowup $F^{(i)}_T$ of
$F^{(i-1)}_T$ at $t_{i-1}$.  (If each $t_i$ is, in fact, a $T$-point of
the exceptional divisor $E^{(i)}_T$ of $F^{(i)}_T$, then
$(t_0,\dots,t_n)$ is a sequence of infinitely near points in the
traditional sense.)  The sequences of arbitrarily near $T$-points form a
functor in $T$, and it is representable by a smooth $Y$-scheme
$F^{(n)}$, according to Proposition~\ref{prp-Har} below; this result is
due, in essence, to Harbourne \cite[Prp.~I.2, p.~104]{Hb86}.

 We say that the sequence $(t_0,\dots,t_n)$ is {\it strict} if, for each
$i,j$ with $1\le j\le i$, the image $T^{(i)}\subset F^{(i)}_T$ of $t_i$
is either (a) disjoint from, or (b) contained in, the strict transform
of the exceptional divisor $E^{(j)}_T$ of $F^{(j)}_T$.  If (b) obtains,
then we say that $t_i$ is {\it proximate\/} to $t_j$ and we write
$t_i\succ t_j$.

To each strict sequence, we associate, in Section~3, an unweighted
Enriques diagram $\I U$ and an ordering $\theta\:\I U\risom \{0, \dotsc,
n\}$.  Effectively, $\I U$ is just a graph whose vertices are the $t_i$.
There is a directed edge from $t_j$ to $t_i$ provided that $j+1\le i$
and that the map from $F^{(i)}_T$ to $F^{(j+1)}_T$ is an isomorphism in a
neighborhood of $T^{(i)}$ and embeds $T^{(i)}$ in $E^{(j+1)}_T$.  In
addition, $\I U$ inherits the binary relation of proximity.  Finally,
$\theta$ is defined by $\theta(t_i):=i$.  This material is discussed in
more detail in Section~2.  In particular, to aid in passing from
$(t_0,\dots,t_n)$ to $(\I U,\,\theta)$, we develop a new combinatorial
notion, which we call a \textit{proximity structure}.

  Different strict sequences often give rise to isomorphic pairs $(\I
U,\,\theta)$.  If we fix a pair, then the corresponding sequences form a
functor, and it is representable by a subscheme $F(\I U,\,\theta)$ of
$F^{(n)}$, which is $Y$-smooth with irreducible geometric fibers of a
certain dimension.  This statement is asserted by Theorem~\ref{thm:3-2},
which was inspired by Ro\'e's Proposition 2.6 in \cite{Roe01}.

Given another ordering $\theta'$, in Section~4 we construct a natural
isomorphism
$$\Phi_{\theta,\theta'}\:F(\I U,\,\theta)\risom F(\I
 U,\,\theta').$$
It is easy to describe $\Phi_{\theta,\theta'}$ on geometric points.  A
geometric point of $F(\I U,\,\theta)$ corresponds to a certain sequence
of local rings in the function field of the appropriate geometric
fiber of $F/Y$.  Then $\theta'\circ\theta^{-1}$ yields a suitable
permutation of these local rings, and so a geometric point of $F(\I
U,\,\theta')$.  However, it is  harder to work with arbitrary $T$-points.
Most of the work is carried out in the proofs of Lemmas~\ref{lemL2}
and~\ref{lemAut}, and the work is completed in the proof of
Proposition~\ref{prpIso}.

We easily derive two corollaries.  Corollary~\ref{corAut} asserts that
$\Aut(\I U)$ acts freely on $F(\I U,\,\theta)$; namely, $\gamma\in
\Aut(\I U)$ acts as $\Phi_{\theta,\theta'}$ where
$\theta':=\theta\circ\gamma$.  Corollary~\ref{corQt} asserts that
$\Psi\:F(\I U,\,\theta)\big/{\Aut} (\I U)$ is $Y$-smooth with
irreducible geometric fibers.

A different treatment of $F(\I U,\,\theta)$ is given by A.-K. Liu in
\cite{Liu00}.  In Section~3 on pp.\,400--401, he constructs $F^{(n)}$.
In Subsection~4.3.1 on pp.\,412--414, he discusses his version of an
Enriques diagram, which he calls an ``admissible graph.''  In
Subsections~4.3.2, 4.4.1, and 4.4.2 on pp.\,414--427, he constructs
$F(\I U,\,\theta)$, and proves it is smooth.  In Subsection~4.5 on
pp.\,428--433, he constructs the action of $\Aut(\I U)$ on $F(\I
U,\,\theta)$.  Of course, he uses different notation; also, he doesn't
represent functors.  But, like the present authors, he was greatly
inspired by Vainsencher's approach in \cite{V95} to enumerating the
singular curves in a linear system on a smooth surface.

Our main result is Theorem \ref{thmF2H}.  It concerns the Enriques
diagram $\I D$ obtained by equipping the vertices $V\in \I U$ with
weights $m_V$ satisfying the {\it Proximity Inequality,}
$m_V\ge\tsum_{W\succ V}m_W.$ We discuss the theory of such $\I D$ in
Section~2.  Note that $\Aut(\I D)\subset\Aut(\I U)$.  Set
$d:=\tsum_{V}\binom{\Wt(V)+1}{2}$.  Theorem~\ref{thmF2H} asserts the
existence of a universally injective map from the quotient to the
Hilbert scheme
$$\Psi\:F(\I U,\,\theta)\big/{\Aut} (\I D)\to\Hilb_{F/Y}^d.$$

Proposition~\ref{prpF2H} implies that $\Psi$ factors into a finite map
followed by an open embedding.  So $\Psi$ is an embedding in
characteristic 0.  However, in any positive characteristic, $\Psi$ can
be ramified everywhere; examples are given in Appendix~B, whose content
is due to Tyomkin.  Nevertheless, according to Proposition~\ref{prpM},
in the important case where every vertex of $\I D$ is a root, $\Psi$ is
an embedding in any characteristic.  Further, adding a nonroot does not
necessarily mean there is a characteristic in which $\Psi$ ramifies, as
other examples in Appendix~B show.

We construct $\Psi$ via a relative version of the standard construction
of the complete ideals on a smooth surface over a field, which grew out
of Zariski's work in 1938; the standard theory is reviewed in
Subsection~5.1.  Now, a $T$-point of $F(\I U,\,\theta)$ represents a
sequence of blowing-ups $F^{(i)}_T\to F^{(i-1)}_T$ for $1\le i\le n+1$.
On the final blowup $F^{(n+1)}_T$, for each $i$, we form the preimage of
the $i$th center $T^{(i)}$.  This preimage is a divisor; we multiply it
by $m_{\theta^{-1}(i)}$, and we sum over $i$.  We get an effective
divisor.  We take its ideal, and push down to $F_T$.  The result is an
ideal, and it defines the desired $T$-flat subscheme of $F_T$.  The
flatness holds and the formation of the subscheme commutes with base
change owing to the generalized property of exchange proved in
Appendix~A.  Appendix~A is of independent interest.

It is not hard to see that $\Psi$ is injective on geometric points, and
that its image is the subset $H(\I D)\subset \Hilb^d_{F/Y}$
parameterizing complete ideals with diagram $\I D$ on the fibers of
$F/Y$.  To prove that $\Psi$ induces a finite map onto $H(\I D)$, we use
a sort of valuative criterion; the work appears in Lemma~\ref{lemDVR1}
and Proposition~\ref{prpF2H}.  An immediate corollary,
Corollary~\ref{corHLC}, asserts that $H(\I D)$ is locally closed.  This
result was proved for complex analytic varieties by
Lossen~\cite[Prp.~2.19, p.~35]{Los98} and for excellent schemes by
Nobile and Villamayor~\cite[Thm.~2.6, p.~250]{NV97}.  Their proofs are
rather different from each other and from ours.

In \cite{R03} and \cite{R04}, Russell studies sets somewhat similar to
the $H(\I D)$.  They parameterize isomorphism classes of finite
subschemes of $F$ 
supported at one point.

In short, Section~2 treats weighted and unweighted Enriques diagrams and
proximity structures.  Section~3 treats sequences of arbitrarily near
$T$-points.  To certain ones, the \textit{strict} sequences, we
associate an unweighted Enriques diagram $\I U$ and an ordering
$\theta$.  Fixing $\I U$ and $\theta$, we obtain a functor, which we
represent by a smooth $Y$-scheme $F(\I U,\,\theta)$.  Section~4 treats
the variance in $\theta$.  We produce a free action on $F(\I
U,\,\theta)$ of $\Aut(\I U)$.  Section~5 treats the Enriques diagram $\I
D$ obtained by equipping $\I U$ with suitable weights.  We construct a
map $\Psi$ from $F(\I U,\,\theta)\big/{\Aut} (\I D)$ to $\Hilb_{F/Y}$,
whose image is the locus $H(\I D)$ of complete ideals.  We prove $H(\I
D)$ is locally closed.  Our main theorem asserts that $\Psi$ is
universally injective, and in fact, in characteristic 0, an embedding.
Appendix~A treats the generalized property of exchange used in
constructing $\Psi$.  Finally, Tyomkin's Appendix~B treats a few
examples: in some, $\Psi$ is ramified; in others, there's a nonroot, yet
$\Psi$ is unramified.

\section{Enriques diagrams}\label{sc:diag} 
In 1915, Enriques \cite[IV.I, pp.~350--51]{EC15} explained a way to
represent the equisingularity type of a plane curve singularity by means
of a directed graph: each vertex represents an arbitrarily near point,
and each edge connects a vertex representing a point to a vertex
representing a point in its first-order neighborhood; furthermore, the
graph is equipped with a binary relation representing the ``proximity'' of
arbitrarily near points.  These graphs have, for a long time, been called
{\it Enriques diagrams}, and in 2000, they were given a modern treatment
by Casas in \cite[Sec.~3.9, pp.~98--102]{Ca00}.

Based in part on a preliminary edition of Casas' monograph, a more
axiomatic treatment was given by the authors in \cite[\S~2]{K--P}, and
this treatment is elaborated on here in Subsection~\ref{sb:basics}.  In
this treatment, the vertices are weighted, and the number of vertices is
minimized.  When the diagram arises from a curve, the vertices correspond
to the ``essential points'' as defined by Greuel et
al. \cite[Sec.~2.2]{GLS98}, and the weights are the multiplicities of
the points on the strict transforms.   Casas' treatment is similar: the
Proximity Inequality is always an equality, and the leaves, or extremal
vertices, are of weight 1; so the rest of the weights are determined.

At times, it is convenient to work with unweighted diagrams.  For this
reason, Ro\'e \cite[\S1]{Roe01}, inspired by Casas, defined an
``Enriques diagram'' to be an unweighted graph, and he imposed five
conditions, which are equivalent to our Laws of Proximity and of
Succession.  Yet another description of unweighted Enriques diagrams is
developed below in Subsection~\ref{sbs-ps} and Proposition~\ref{prp-ps}
under the name of ``proximity structure.''  This description facilitates
the formal assignment, in Subsection~\ref{sb:curve}, of an Enriques
diagram to a plane curve singularity.  Similarly, the description
facilitates the assignment in Section~\ref{sc:Inp} of the Enriques
diagram associated to a strict sequence of arbitrarily near points.

At times, it is convenient to order the elements of the set underlying
an Enriques diagram or underlying a proximity structure.  This subject
is developed in Subsections \ref{sb-Ud} and \ref{sbs-ps} and in
Corollary~\ref{cor-ps}.  It plays a key role in the later sections.

Finally, in Subsection~\ref{sb:nchar}, we discuss several useful
numerical characters.  Three were introduced in \cite[Sct.~2,
p.~214]{KP99}, and are recalled here.  Proposition~\ref{pr:2-7}
describes the change in one of the three when a singularity is
blown up; this result is needed in \cite{KPnpc}

\begin{sbs}[Enriques diagrams] \label{sb:basics}
First, recall some general notions.  In a directed graph, a vertex $V$
is considered to be one of its own predecessors and one of its own
successors.  Its other predecessors and successors $W$ are said to be
{\it proper}.  If there are no loops, then $W$ is said to be {\it
remote\/}, or {\it distant}, if there is a distinct third vertex lying
between $V$ and $W$; otherwise, then $W$ is said to be {\it immediate}.

A {\it tree\/} is a directed graph with no loops; by definition, it has
a single initial vertex, or {\it root,} and every other vertex has a
unique immediate predecessor.  A final vertex is called a {\it leaf}.  A
disjoint union of trees is called a {\it forest}.

Next, from \cite[\S~2]{K--P}, recall the definition of a minimal
Enriques diagram.  It is a finite forest $\I D$ with additional
structure.  Namely, each vertex $V$ is assigned a weight $\Wt(V)$, which
is an integer at least 1. Also, the forest is equipped with a binary
relation; if one vertex $V$ is related to another $U$, then we say that
$V$ is {\it proximate\/} to $U$, and write $V \succ U$.  If $U$ is a
remote predecessor of $V$, then we call $V$  a {\it satellite of}
$U$; if not, then we say $V$ is {\it free}.  Thus a root is free, and a
leaf can be either free or a satellite.

Elaborating on \cite{K--P}, call $\I D$ an {\it Enriques diagram\/} if
$\I D$ obeys these three laws:

{\smallskip\par
 \advance\leftskip by 2\parindent\it\parindent=0pt
 \(Law of Proximity)\enspace
 A root is proximate to no vertex.  If a vertex is not a root, then it
is proximate to its immediate predecessor and to at most one other
vertex; the latter must be a remote predecessor.  If one vertex is
proximate to a second, and if a distinct third lies between the two,
then it too is proximate to the second.

 \smallskip

 \(Proximity Inequality)\enspace  For each vertex $V$,
        $$m_V\ge\tsum_{W\succ V}m_W.$$

 \smallskip

 \(Law of Succession)\enspace A vertex may have any number of\/ {\rm free}
immediate successors, but at most two immediate successors
may be satellites, and they must be satellites of different vertices.
 \smallskip\par
}

Notice that, by themselves, the Law of Proximity and the Proximity
Inequality imply that a vertex $V$ has at most $m_V$ immediate
successors; so, although this property is included in the statement of
the Law of Succession in  \cite[\S~2]{K--P}, it is omitted here.

Recovering the notion in \cite{KP99}, call an Enriques diagram $\I D$
{\it minimal\/} if $\I D$ obeys the following fourth law:
 {\smallskip\par
 \advance\leftskip by 2\parindent\it\parindent=0pt
 \(Law of Minimality)\enspace There are only finitely many vertices, and
every leaf of weight $1$ is a satellite.
 \smallskip\par
}
 \noindent In \cite{KP99}, the Law of Minimality did not include the
present finiteness restriction; rather, it was imposed at the outset.
\end{sbs}

\begin{sbs}[Unweighted diagrams] \label{sb-Ud}
In \cite[\S1]{Roe01}, Ro\'e defines an Enriques diagram to be an {\it
unweighted\/} finite forest that is equipped with a binary relation,
called ``proximity,'' that is required to satisfy five conditions.  It
is not hard to see that his conditions are equivalent to our Laws of
Proximity and Succession.  Let us call this combinatorial structure an
{\it unweighted Enriques diagram}.

Let $\I U$ be any directed graph on $n+1$ vertices.  By an
{\it ordering\/} of $\I U$, let us mean a bijective mapping
	$$\theta\:\I U\risom\{0,\dotsc,n\}$$
such that, if one vertex $V$ precedes another $W$, then $\theta(V)\le
\theta(W)$.  Let us call the pair $(\I U,\,\theta)$ an {\it ordered
directed graph}.

An ordering $\theta$ need not be unique.  Furthermore, if one exists,
then plainly $\I U$ has no loops.  Conversely, if $\I U$ has no
loops\emdash if it is a forest\emdash then $\I U$ has at least one
ordering.  Indeed, then $\I U$ has a leaf $L$.  Let $\I T$ be the complement
of $L$ in $\I U$.  Then $\I T$ inherits the structure of a forest.
So, by induction on $n$, we may assume that $\I T$ has an ordering.
Extend it to $\I U$ by mapping $L$ to $n$.

Associated to any  ordered unweighted Enriques diagram $(\I U,\,\theta)$
is its {\it proximity matrix\/}  $(p_{ij})$, which  is the  $n+1$
by  $n+1$ lower
triangular matrix defined by
	$$p_{ij}:=\begin{cases}
	  \hfill1, &\text{if } i=j;\\
              -1, &\text{if $\theta^{-1}i$ is proximate to }\theta^{-1}j;\\
	  \hfill0, &\text{otherwise.}
	  \end{cases}$$
The transpose was introduced by Du Val in 1936, and he named it the
``proximity matrix'' in 1940; Lipman \cite[p.~298]{Li94} and others have
followed suit.  The definition here is the one used by Ro\'e
\cite{Roe01} and Casas \cite[p.~139]{Ca00}.

Note that $(\I U,\,\theta)$ is determined up to unique isomorphism by
$(p_{ij})$.
\end{sbs}

\begin{sbs}[Proximity structure]\label{sbs-ps}
Let $\I U$ be a finite set equipped with a binary relation.  Call $\I U$
a {\it proximity structure,} its elements {\it vertices,} and the
relation {\it proximity\/} if the following three laws are obeyed:
 {\smallskip\par
 \advance\leftskip by 3\parindent\parindent=0pt\it
\noindent\llap{\(P1)\enspace}No vertex is proximate to itself; no two
vertices are each proximate to the other.

\noindent\llap{\(P2)\enspace}Every vertex is proximate to at most two
others; if to two, then one of the two is proximate to the other.

\noindent\llap{\(P3)\enspace}Given two vertices, at most one other is
proximate to them both.
 \smallskip\par
}
A proximity structure supports a natural structure of directed graph.
Indeed, construct an edge proceeding from one vertex $V$ to another $W$
whenever either $W$ is proximate only to $V$ or $W$ is proximate both to
$V$ and $U$ but $V$ is proximate to $U$ (rather than $U$ to $V$).  Of
course, this graph may have loops; for example, witness a triangle with
each vertex proximate to the one clockwise before it, and witness a
pentagon with each vertex proximate to the two clockwise before it.

Let us say that a proximity structure is {\it ordered\/} if its vertices
are numbered, say $V_0,\dotsc,V_n$, such that, if $V_i$ is proximate to
$V_j$, then $i>j$.
\end{sbs}

\begin{prp}\label{prp-ps}
The unweighted Enriques diagrams sit in natural bijective correspondence
with the proximity structures whose associated graphs have no loops.
\end{prp}
\begin{proof}
First, take an unweighted Enriques diagram, and let's check that its
proximity relation obeys Laws (P1) to (P3).  

A vertex is proximate only to a proper successor; so no vertex is
proximate to itself.  And, if two vertices were proximate to one
another, then each would succeed the other; so there would be a loop.
Thus (P1) holds.

A root is proximate to no vertex.  Every other vertex $W$ is proximate
to its immediate predecessor $V$ and to at most one other vertex $U$,
which must be a remote predecessor.  Since an immediate predecessor is
unique in a forest, $V$ must lie between $W$ and $U$; whence, $V$ must
be proximate to $U$.  Thus (P2) holds.

Suppose two vertices $W$ and $X$ are each proximate to two others $U$
and $V$.  Say $V$ is the immediate predecessor of $W$.  Then $U$ is a
remote predecessor of $W$; so $U$ precedes $V$.  Hence $V$ is also the
immediate predecessor of $X$, and $W$ is also a remote predecessor of
$X$.  Thus both $W$ and $X$ are immediate successors of $V$, and both
are satellites of $W$; so the Law of Succession is violated.  Thus (P3)
holds.

Conversely, take a proximity structure whose associated graph has no
loops.  Plainly, a root is proximate to no vertex.  Suppose a vertex $W$
is not a root.  Then $W$ has an immediate predecessor $V$.  Plainly, $W$
is proximate to $V$.  Plainly, $W$ is proximate to at most one other
vertex $U$, and if so, then $V$ is proximate to $U$.  Since $U$ cannot
also be proximate to $V$ by (P1), it follows that $V$ is the only
immediate predecessor to $W$.

Every vertex is, therefore, preceded by a unique root.  Plainly the
connected component of each root is a tree.  Thus the graph is a finite
forest.

Returning to $U$, $V$, and $W$, we must show that $U$ precedes $W$.
Now, $V$ is proximate to $U$.  So $V$ is not a root.  Hence $V$ has an
immediate predecessor $V'$.  If $V'= U$, then stop.  If not, then $V'$
is proximate to $U$ owing to the definition of the associated graph,
since $V$ is proximate to $U$.  Hence, similarly, $V'$ has an immediate
predecessor $V''$.  If $V''= U$, then stop.  If not, then repeat the
process.  Eventually, you must stop since the number of vertices is
finite.  Thus $U$ precedes $W$.  Furthermore, every vertex between $U$
and $W$ is proximate to $U$.  Thus the Law of Proximity holds.

Continuing with $U$, $V$, and $W$, suppose that $W'$ is a second
immediate successor of $V$ and that $W'$ is also proximate to a vertex
$U'$.  Then $U'\neq U$ since at most one vertex can be proximate to both
$V$ and $U$ by (P3).

  Finally, suppose that $W''$ is a third immediate
successor of $V$ and that $W''$ is also proximate to a vertex $U''$.
Then $U''\neq U$ and $U''\neq U'$ by what we just proved.  But $V$ is
proximate to each of $U$, $U'$, and $U''$.  So (P2) is violated.  Thus
the Law of Succession holds, and the proof is complete.
\end{proof}

\begin{cor}\label{cor-ps}
The ordered unweighted Enriques diagrams sit in natural bijective
correspondence with the ordered proximity structures.
\end{cor}
\begin{proof}
Given an unweighted Enriques diagram, its proximity relation
obeys Laws (P1) to (P3) by the proof of Proposition~\ref{prp-ps}.  And,
if one vertex $V$ is proximate to another $W$, then $W$ precedes $V$.
So $\theta(W)<\theta(V)$ for any ordering $\theta$.  Hence, if $V$ is
numbered $\theta(V)$ for every $V$, then the proximity structure is
ordered.

Conversely, take an ordered proximity structure.  The associated
directed graph is, plainly, ordered too, and so has no loops.
And, the Laws of Proximity and Succession hold by the proof of
Proposition~\ref{prp-ps}.  Thus the corollary holds.
\end{proof}

\begin{sbs}[Numerical characters] \label{sb:nchar} In \cite[Sct.~2,
p.~214]{KP99}, a number of numerical characters were introduced, and
three of them are useful in the present work.

The first character makes sense for any unweighted Enriques
diagram $\I U$, although it was not defined in this generality before;
namely, the {\it dimension\/} $\dim(\I U)$ is the number of roots plus
the number of free vertices in $\I U$, including roots.  Of course, the
definition makes sense for
a weighted Enriques diagram $\I D$; namely, the {\it dimension\/}
$\dim(\I D)$ is simply the dimension of the underlying unweighted
diagram.

The second and third characters make sense only for a weighted Enriques
diagram $\I D$; namely, the {\it degree\/} and {\it codimension\/} are
defined by the formulas
  \begin{align*}
        \deg(\I D)&:=\tsum_{V\in\I D}\binom{\Wt(V)+1}{2};\\
        \cod(\I D)&:=\deg(\I D)-\dim(\I D).
  \end{align*}

It is useful to introduce a new character, the {\it type\/} of a vertex
$V$ of  $\I U$ or of  $\I V$.  It is defined by the formula
        $$\type(V):=\begin{cases}0,&\text{if $V$ is a satellite;}\\
                        1,&\text{if $V$ is a free vertex, but not a root;}\\
                        2,&\text{if $V$ is a root.}
                     \end{cases}$$
The type appears in the following two formulas:
 \begin{align}
 \dim(\I A)&=\tsum_{V\in\I A}\type(V);
 \label{eq:2-1}\\
 \cod(\I A)&=\tsum_{V\in\I A}\bigl[\binom{\Wt(V)+1}2-\type(V)\bigr].
 \label{eq:2-2}
 \end{align}
Formula~\ref{eq:2-2} is useful because every summand is nonnegative in
general and positive
when $\I A$ is a minimal Enriques diagram.
 \end{sbs}

\begin{sbs}[The diagram of a curve] \label{sb:curve} Let $C$ be a
reduced curve lying on a smooth surface over an algebraically closed
ground field; the surface need not be complete.  In \cite[Sec.~2,
p.~213]{KP99} and again in \cite[Sec.~2, p.~72]{K--P}, we stated that,
to $C$, we can associate a minimal Enriques diagram $\I D$.  (It
represents the equisingularity type of $C$; this aspect of the theory is
treated in \cite[p.~99]{Ca00} and \cite[pp.~543--4]{GLS98}.)  Here is
more explanation about the construction of $\I D$.

First, form the configuration of all arbitrarily near points of the
surface lying on all the branches of the curve through all its singular
points.  Say that one arbitrarily near point is {\it proximate\/} to a
second if the first lies above the second and on the strict transform of
the exceptional divisor of the blowup centered at the second.  Then Laws
(P1) to (P3) hold because three strict transforms never meet and, if two
meet, then they meet once and transversely.  Plainly, there are no
loops.  Hence, by Proposition~\ref{prp-ps}, this configuration is an
unweighted Enriques diagram.

Second, weight each arbitrarily near point with its multiplicity as a
point on the strict transform of the curve.  By the theorem of strong
embedded resolution, all but finitely many arbitrarily near points are of
multiplicity 1, and are proximate only to their immediate predecessors;
prune off all the infinite unbroken successions of such points, leaving
finitely many points.  Then the Law of Minimality holds.

Finally, the Proximity Inequality holds for this well-known reason: the
multiplicity of a point $P'$ on a strict transform $C'$ can be computed
as an intersection number $m$ on the blowup at $P'$ of the surface
containing $C'$; namely, $m$ is the intersection number of the
exceptional divisor and the strict transform of $C'$; the desired
inequality results now from Noether's formula for $m$ in terms of
multiplicities of arbitrarily near points.  (In \cite[p.~83]{Ca00}, the
inequality is an equality, because no pruning is done.)  Therefore, this
weighted configuration is a minimal Enriques diagram.  It is $\I D$.

Notice that, if $K$ is any algebraically closed extension field of the
ground field, then the curve $C_K$ also has diagram  $\I D$.
\end{sbs}

\begin{prp}\label{pr:2-7} Let $C$ be a reduced curve lying on a smooth
surface over an algebraically closed field.  Let $\I D$ be the minimal
Enriques diagram of $C$, and $P\in C$ a singular point of multiplicity
$m$.  Form the blowup of the surface at $P$, the exceptional divisor
$E$, the proper transform $C'$ of $C$, and the union $C'':=C'\cup E$.
Let $\I D'$ be the diagram of $C'$, and $\I D''$ that of $C''$.  Then
        $$\textstyle\cod(\I D)-\cod(\I D')\ge\binom{m+1}2-2
        \text{ and }\cod(\I D)-\cod(\I D'')=\binom m2-2;$$
  equality holds in the first relation if and only if  $P$ 
is an ordinary $m$-fold point.
 \end{prp}
\begin{proof} We obtain $\I D'$ from $\I D$ by deleting the root $R$
corresponding to $P$ and also all the vertices $T$ that are of weight 1,
proximate to $R$, and such that all successors of $T$ are also (of
weight 1 and) proximate to $R$ (and so deleted too).  Note that an
immediate successor of $R$ is free; if it is deleted, then it has weight
1, and if it is not deleted, then it becomes a root of $\I D'$.  Also,
by the Law of Proximity, an undeleted satellite of $R$ becomes
a free vertex of $\I D'$.

Let $\sigma$ be the total number of satellites of $R$, and $\rho$ the
number of undeleted immediate successors.  Then it follows from
the  Formula~(\ref{eq:2-2}) that
    $$\textstyle\cod(\I D)-\cod(\I D')=\binom{m+1}2-2+\sigma+\rho.$$
Thus the asserted inequality holds, and it is an equality if and
only if $\sigma=0$ and $\rho=0$.  So it is an equality if $P$ is an
ordinary $m$-fold point.

Conversely, suppose $\sigma=0$ and $\rho=0$.  Then $R$ has no immediate
successor $V$ of weight 1 for the following reason.  Otherwise, any
immediate successor $W$ of $V$ is proximate to $V$ by the Law of
Proximity.  So $W$ has weight 1 by the Proximity Inequality.  Hence, by
recursion, we conclude that $V$ is succeeded by a leaf $L$ of weight 1.
So, by the Law of Minimality, $L$ is a satellite.  But $\sigma=0$.
Hence $V$ does not exist.  But $\rho=0$.  Hence $R\/$ has no successors
whatsoever.  So $P$ is an ordinary $m$-fold point.

Furthermore, we obtain $\I D''$ from $\I D$ by deleting $R$ and by
adding 1 to the weight of each $T$ proximate to $R$.  So a satellite of
$R$ becomes a free vertex of $\I D''$, and an immediate successor of $R$
becomes a root of $\I D''$.  In addition, for each smooth branch of $C$
that is transverse at $P$ to all the other branches, we adjoin an
isolated vertex (root) of weight 2.

The number of adjoined vertices is $m-\tsum_{T\succ R}\Wt(T)$.  So, by
Formula~(\ref{eq:2-2}),
  \begin{multline*}
 \textstyle\cod(\I D)-\cod(\I D'')=\binom{m+1}2-2
        +\sum_{T\succ R}\bigl[\binom{\Wt(T)+1}2-\type(T)\bigr]\hfill\\
  \hfill\textstyle -\sum_{T\succ R}\bigl[\binom{\Wt(T)+2}2-
  (\type(T)+1)\bigr]-\bigl[m-\tsum_{T\succ R}\Wt(T)\bigr].
 \end{multline*}
 The right hand side reduces to $\binom m2-2$.  So the asserted equality
 holds.
 \end{proof}

\section{Infinitely near points}\label{sc:Inp}
 Fix a smooth family of geometrically irreducible surfaces $\pi\:F\to Y$.
 In this section, we study sequences of arbitrarily near $T$-points of
 $F/Y$.  They are defined in Definition~\ref{dfn:INP}.  Then
 Proposition~\ref{prp-Har} asserts that they form a representable
 functor.  In essence, this result is due to Harbourne \cite[Prp.~I.2,
 p.~104]{Hb86}, who identified the functor of points of the iterated
 blow-up that was introduced in \cite[Sct.~4.1, p.~36]{Kl81} and is
 recalled in Definition~\ref{dfn:DerFam}. 

 In the second half of this section, we study a special kind of sequence
 of arbitrarily near $T$-points, the {\it strict} sequence, which is
 defined in Definition~\ref{dfn:sseq}.  To each strict sequence is
 associated a natural ordered unweighted Enriques diagram owing to
 Propositions~\ref{prp-assdiag} and \ref{prp-ps}.  Finally,
 Theorem~\ref{thm:3-2} asserts that the strict sequences with given
 diagram $(\I U,\,\theta)$ form a functor, which is representable by a
 $Y$-smooth scheme with irreducible geometric fibers of dimension
 $\dim(\I U)$.  This theorem was inspired by Ro\'e's Proposition 2.6 in
 \cite{Roe01}.

 \begin{dfn}\label{dfn:DerFam}
  By induction on $i\ge0$, let us define more families
   $$\pi^{(i)}\:F^{(i)}\to F^{(i-1)},$$
 which are like $\pi\:F\to Y$.  Set $\pi^{(0)}:=\pi$.  Now, suppose
 $\pi^{(i)}$ has been defined.  Form the fibered product of $F^{(i)}$
 with itself over $F^{(i-1)}$, and blow up along the diagonal
 $\Delta^{(i)}$.  Take the composition of the blowup map and the second
 projection to be $\pi^{(i+1)}$.

 In addition,  for $i\ge1$, let $\vf^{(i)}\: F^{(i)} \to F^{(i-1)}$ be
 the composition of the blowup map and the first projection, and let
 $E^{(i)}$ be the exceptional divisor.  Finally,
 set $\vf^{(0)}:=\pi$; so $\vf^{(0)}=\pi^{(0)}$.
  \end{dfn}

 \begin{lem}\label{lem:DerFam}
  Both $\pi^{(i)}$ and $\vf^{(i)}$ are smooth, and have geometrically
 irreducible fibers of dimension $2$.  Moreover, $E^{(i)}$ is equal, as
 a polarized scheme, to the bundle $\bb P(\Omega^1_{\pi^{(i-1)}})$ over
 $F^{(i-1)}$, where $\Omega^1_{\pi^{(i-1)}}$ is the sheaf of relative
 differentials.
  \end{lem}
 \begin{proof}
   The first assertion holds for $i=0$ by hypothesis.  Suppose it holds
   for $i$.  Consider the fibered product formed in
   Definition~\ref{dfn:DerFam}.  Then both projections are smooth, and
   have geometrically irreducible fibers of dimension $2$; also, the
   diagonal $\Delta^{(i)}$ is smooth over both factors.  It follows that
   the first assertion holds for $i+1$.

 The second assertion holds because $\Omega^1_{\pi^{(i-1)}}$ is the
 conormal sheaf of $\Delta^{(i)}$.
  \end{proof}

 \begin{dfn}\label{dfn:INP}
  Let $T$ be a  $Y$-scheme.  Given a sequence of blowups
  \begin{equation*}\label{eq:3-1.1}
         F^{(n+1)}_T\xrightarrow{\vf^{(n+1)}_T} F^{(n)}_T\to\dotsb
         \to F'_T\xrightarrow{\vf^{(1)}_T}  F_T:=F\x_YT
  \end{equation*}
  whose $i$th center $T^{(i)}\subset F^{(i)}_T$ is the image of a section
 $t_{i}$ of $F^{(i)}_T/T$ for $0\le i\le n$, call $(t_0,\dots,t_n)$ a
 {\it sequence of arbitrarily near $T$-points\/} of $F/Y$.

 For $1\le i\le n+1$, denote the exceptional divisor in $F^{(i)}_T$ by
 $E^{(i)}_T$.
  \end{dfn}

 The following result is a version of Harbourne's Proposition~I.2 in
 \cite[p.~104]{Hb86}.

 \begin{prp}[Harbourne]\label{prp-Har}
  As $T$ varies, the sequences $(t_0,\dots,t_n)$ of arbitrarily near
 $T$-points of $F/Y$ form a functor, which is represented by
 $F^{(n)}/Y$.

 Given $(t_0,\dots,t_n)$ and $i$, say $(t_0,\dots,t_i)$ is represented by
 $\tau_{i}\:T\to F^{(i)}$.  Then $\pi^{(i)}\tau_{i}=\tau_{i-1}$ where
 $\tau_{-1}$ is the structure map.  Also,
 $F^{(i+1)}_T=F^{(i+1)}\x_{F^{(i)}}T$ where $F^{(i+1)}\to {F^{(i)}}$ is
 $\pi^{(i+1)}$; correspondingly, $t_{i}=(\tau_{i},1)$ and
 $E^{(i+1)}_T=E^{(i+1)}\x_{F^{(i)}}T$; moreover, $T^{(i)}$ is the
 scheme-theoretic image of $E^{(i+1)}_T$ under
 $\vf^{(i+1)}_T\:F^{(i+1)}_T\to F^{(i)}_T$.  Finally, $\vf^{(i+1)}_T$ is
 induced by $\vf^{(i+1)}$, and $F^{(i+1)}_T\to T$ is induced by
 $\pi^{(i+1)}$.
  \end{prp}
  \begin{proof}
  First, observe that, given a section of any smooth map $a\:A\to B$,
 blowing up $A$ along the section's image, $C$ say, commutes with
 changing the base $B$.  Indeed, let $\mc I$ be the ideal of $C$, and for
 each $m\ge0$, consider the exact sequence
    $$0\to \mc I^{m+1}\to \mc I^m \to \mc I^m \big/\mc I^{m+1}\to 0.$$
  Since $a$ is smooth, $\mc I^m \big/\mc I^{m+1}$ is a locally free $\mc
 O_C$-module, so $B$-flat.  Hence forming the sequence commutes with
 changing $B$.  However, the blowup of $A$ is just $\Proj \bigoplus_m\mc I^m$.
 Hence forming it commutes too.

 Second, observe in addition that $C$ is the scheme-theoretic image of the
 exceptional divisor, $E$ say, of this blowup.  Indeed, this image is the
 closed subscheme of $C$ whose ideal is the kernel of the comorphism of
 the map $E\to C$.  However, this comorphism is an isomorphism, because
 $E= \bb P(I/I^2)$ since $a$ is smooth.

 The first observation implies that the sequences
 $(t_0,\dots,t_n)$ form a functor, because, given any $Y$-map $T'\to
 T$, each induced map
         $$F^{(i+1)}_T\x_TT'\to F^{(i)}_T\x_TT'$$
  is therefore the blowing-up along the image of the induced section of
 $F^{(i)}_T\x_TT'\big/T'$.

 To prove this functor is representable by $F^{(n)}/Y$, we must set up a
 functorial bijection between the sequences $(t_0,\dots,t_n)$ and the
 $Y$-maps $\tau_{n}\:T\to F^{(n)}$.  Of course, $n$ is arbitrary.  So
 $(t_0,\dots,t_i)$ then determines a $Y$-map $\tau_{i}\:T\to F^{(i)}$,
 and correspondingly we want the remaining assertions of the proposition
 to hold as well.

 So given $(t_0,\dots,t_n)$, let us construct appropriate $Y$-maps
 $\tau_{i}\:T\to F^{(i)}$ for $-1\le i\le n$.  We proceed by induction on
 $i$.  Necessarily, $\tau_{-1}\:T\to Y$ is the structure map, and
 correspondingly, $F^{(0)}_T=F^{(0)}\x_{F^{(-1)}}T$ owing to the
 definitions.

 Suppose we've constructed $\tau_{i-1}$.  Then
 $F^{(i)}_T=F^{(i)}\x_{F^{(i-1)}}T$.  Set $\tau_{i}:=p_1t_{i}$ where
 $p_1\:F^{(i)}_T\to F^{(i)}$ is the projection.  Then
 $\tau_{i-1}=\pi^{(i)}\tau_{i}$.  Also, $t_{i}=(\tau_{i},1)$; so $t_{i}$
 is the pullback, under the map $(1,\tau_{i})$, of the diagonal map of
 $F^{(i)}/F^{(i-1)}$.  Therefore, owing
 to the first observation, $F^{(i+1)}_T=F^{(i+1)}\x_{(F^{(i)}\x
 F^{(i)})}F^{(i)}_T$ where $F^{(i)}_T\to F^{(i)}\x_{F^{(i-1)}} F^{(i)}$
 is equal to $1\x\tau_i$.  Hence $F^{(i+1)}_T=F^{(i+1)}\x_{F^{(i)}}T$
 where $F^{(i+1)}\to {F^{(i)}}$ is $\pi^{(i+1)}$.  It follows formally
 that $E^{(i+1)}_T=E^{(i+1)}\x_{F^{(i)}} T$, that $F^{(i+1)}_T\to
 F^{(i)}_T$ is induced by $\vf^{(i+1)}$, and that $F^{(i+1)}_T\to T$ is
 induced by $\pi^{(i+1)}$.

 By the second observation above, $T^{(i)}$ is the scheme-theoretic
 image of $E^{(i+1)}_T$.

 Conversely, given a map $\tau_{n}\:T\to F^{(n)}$, set
 $\tau_{i-1}:=\pi^{(i)}\dotsm\pi^{(n)}\tau_{n}$ for $0\le i\le n$; so
 $\tau_{i-1}\:T\to F^{(i-1)}$.  Set $F^{(i)}_T:=F^{(i)}\x_{F^{(i-1)}}T$
 where the map $F^{(i)}\to {F^{(i-1)}}$ is $\pi^{(i)}$ for $0\le i\le
 n+1$.  Then $\tau_{i}$ defines a section $t_{i}$ of $F^{(i)}_T/T$.
 Furthermore, blowing up its image yields the map $F^{(i+1)}_T\to
 F^{(i)}_T$ induced by $\vf^{(i+1)}$, because, as noted above, forming
 the blowup along $\Delta^{(i)}$ commutes with changing the base via
 $1\x\tau_{i}$.  Thus $(t_0,\dots,t_n)$ is a sequence of arbitrarily near
 $T$-points of $F/Y$.

 Plainly, for each $T$, we have set up the bijection we sought, and it is
 functorial in $T$.  Since we have checked all the remaining assertions of
 the proposition, the proof is now complete.
  \end{proof}

 \begin{dfn}\label{dfn:sseq}
  Given a sequence $(t_0,\dots,t_n)$ of arbitrarily near $T$-points of
 $F/Y$, let us call it {\it strict\/} if, for $0\le i\le n$, the image
 $T^{(i)}$ of $t_i$ satisfies the following $i$ conditions, defined by
 induction on $i$.  There are, of course, no conditions on $T^{(0)}$.
 Fix $i$, and suppose, for $0\le j<i$, the conditions on $T^{(j)}$ are
 defined and satisfied.

 The $i$ conditions on $T^{(i)}$ involve the natural embeddings
   $$e^{(j,i)}_T\:E^{(j)}_T\into F^{(i)}_T \text{\quad for } 1\le j\le i,$$
 which we assume defined by induction; see the next paragraph.  (The
 image $e^{(j,i)}_TE^{(j)}_T$ can be regarded as the ``strict
 transform'' of  $E^{(j)}_T$ on $F^{(i)}_T$.) 
 The $j$th
 condition requires  $e^{(j,i)}_TE^{(j)}_T$ either (a) to be
 disjoint from $T^{(i)}$ or (b) to contain $T^{(i)}$ as a subscheme.

 Define $e^{(i+1,\,i+1)}_T$ to be the inclusion.  Now, for $1\le j\le i$,
 we have assumed that $e^{(j,i)}_T$ is defined, and required that its
 image satisfy either (a) or (b).  If (a) is satisfied, then the
 blowing-up $F^{(i+1)}_T\to F^{(i)}_T$ is an isomorphism on a
 neighborhood of $e^{(j,\,i)}_TE^{(j)}_T$, namely, the complement of
 $T^{(i)}$; so then $e^{(j,\,i)}_T$ lifts naturally to an embedding
 $e^{(j,\,i+1)}_T$.  If (b) is satisfied, then $T^{(i)}$ is a relative
 effective divisor on the $T$-scheme $e^{(j,i)}E^{(j)}_T$, because
 $E^{(j)}_T$ and $T^{(i)}$ are flat over $T$, and the latter's fibers are
 effective divisors on the former's fibers, which are $\bb P^1$s; hence,
 then blowing up $e^{(j,i)}_TE^{(j)}_T$ along $T^{(i)}$ yields an
 isomorphism.  But the blowup of $e^{(j,i)}_TE^{(j)}_T$ embeds naturally
 in $F^{(i)}_T$.  Thus, again, $e^{(j,\,i)}_T$ lifts naturally.
  \end{dfn}

 \begin{dfn}\label{dfn:prox}
  Given a strict sequence $(t_0,\dots,t_n)$ of arbitrarily near
 $T$-points of $F/Y$,  say that $t_i$ is {\it proximate\/} to $t_j$
 if $j<i$ and $e^{(j+1,\,i)}E^{(j+1)}_T$ contains $T^{(i)}$.
   \end{dfn}

 \begin{lem}\label{lem-Z}
  Let $(t_0,\dots,t_n)$ be a strict sequence of arbitrarily near
 $T$-points of $F/Y$.  Fix $n+1\ge i\ge j\ge k\ge 1$.  Then
 $\vf^{(j+1)}_T\dotsb\vf^{(i)}_Te^{(k,i)}_T=e^{(k,j)}_T$, and $T^{(j-1)}$
 is the scheme-theoretic image of $e^{(j,i)}_TE^{(j)}_T$ under
 $\vf^{(j)}_T\dotsb\vf^{(i)}_T$.  Set
         $$Z^{(i)}_T:=e^{(k,i)}_TE^{(k)}_T\bigcap e^{(j,i)}_TE^{(j)}_T.$$
 If $j> k$ and $Z^{(i)}_T\neq\emptyset$, then
 $\vf^{(j)}_T\dotsb\vf^{(i)}_T$ induces an isomorphism $Z^{(i)}_T\risom
 T^{(j-1)}$, and $t_{j-1}$ is proximate to $t_{k-1}$; moreover, then
 $Z^{(i)}_T$ meets no $e^{(l,i)}_TE^{(l)}_T$ for $l\neq j,k$.
  \end{lem}
 \begin{proof}
 The formula $\vf^{(j+1)}_T\dotsb\vf^{(i)}_Te^{(k,i)}_T=e^{(k,j)}_T$ is
 trivial if $i=j$.  It holds by construction if $i=j+1$.  Finally, it
 follows by induction if $i>j+1$.  With $k:=j$, this formula implies that
 $E^{(j)}_T$ is the scheme-theoretic image of $e^{(j,i)}_TE^{(j)}_T$
 under $\vf^{(j+1)}_T\dotsb\vf^{(i)}_T$; whence,
 Proposition~\ref{prp-Har} implies that $T^{(j-1)}$ is the
 scheme-theoretic image of $e^{(j,i)}_TE^{(j)}_T$ under
 $\vf^{(j)}_T\dotsb\vf^{(i)}_T$.

 Suppose $j> k$ and $Z^{(i)}_T\neq\emptyset$.  Now, for any $l$ such that
 $i\ge l\ge j$, both $e^{(k,l)}_TE^{(k)}_T$ and $e^{(j,l)}_TE^{(j)}_T$
 are relative effective divisors on $F^{(l)}_T/T$, because they're flat
 and divisors on the fibers.  Hence, on either of $e^{(k,l)}_TE^{(k)}_T$
 and $e^{(j,l)}_TE^{(j)}_T$, their intersection $Z^{(l)}_T$ is a relative
 effective divisor, since each fiber of $Z^{(l)}_T$ is correspondingly a
 divisor.  In fact, each nonempty fiber of $Z^{(l)}_T$ is a reduced point
 on a $\bb P^1$.

 Since $\vf^{(j+1)}_T\dotsb\vf^{(i)}_Te^{(j,i)}_T=e^{(j,j)}_T$ and since
 $e^{(j,j)}_T$ is the inclusion of $E^{(j)}_T$, which is the exceptional
 divisor of the blowing-up $\vf^{(j)}\:F^{(j)}\to F^{(j-1)}$ along
 $T^{(j-1)}$, the map $\vf^{(j)}_T\dotsb\vf^{(i)}_T$ induces a proper map
 $e\:Z^{(i)}_T\to T^{(j-1)}$.  Since the fibers of $e$ are isomorphisms,
 $e$ is a closed embedding.  So since $Z^{(i)}_T$ and $T^{(j-1)}$ are
 $T$-flat, $e$ is an isomorphism onto an open and closed subscheme.

 Since $\vf^{(j)}_T\dotsb\vf^{(i)}_Te^{(k,i)}_T=e^{(k,\,j-1)}_T$, it
 follows that $e^{(k,\,j-1)}_TE^{(k)}_T$ contains a non\-empty subscheme of
 $T^{(j-1)}$.  So since $(t_0,\dots,t_n)$ is strict,
 $e^{(k,\,j-1)}_TE^{(k)}_T$ contains all of $T^{(j-1)}$ as a subscheme.
 Thus $t_{j-1}$ is proximate to $t_{k-1}$.

 It follows that $\vf^{(j)}_T$ induces a surjection $Z^{(j)}_T\onto
 T^{(j-1)}$.  If $i=j$, then this surjection is just $e$, and so $e$ is
 an isomorphism, as desired.

 Suppose $i>j$.  Then $Z^{(j)}_T\bigcap T^{(j)}=\emptyset$.  Indeed,
 suppose not.  Then both $e^{(k,j)}_TE^{(k)}_T$ and $E^{(j)}_T$ meet
 $T^{(j)}$.  So since $(t_0,\dots,t_n)$ is strict, $Z^{(j)}_T$ contains
 $T^{(j)}$ as a closed subscheme.  Both these schemes are $T$-flat, and
 their fibers are reduced points; hence, they coincide.  It follows that
 $e^{(k,\,j+1)}_TE^{(k)}_T$ and $e^{(j,\,j+1)}_TE^{(j)}_T$ are disjoint
 on $ F^{(j+1)}$.  But these subschemes intersect in $Z^{(j+1)}_T$.  And
 $Z^{(j+1)}_T\neq\emptyset$ since $Z^{(i)}_T\neq\emptyset$ and
 $Z^{(i)}_T$ maps into $Z^{(j+1)}_T$.  We have a contradiction, so
 $Z^{(j)}_T\bigcap T^{(j)}=\emptyset$.

 Therefore, $\vf^{(j+1)}_T$ induces an isomorphism $Z^{(j+1)}_T\risom
 Z^{(j)}_T$.  Similarly, $\vf^{(l+1)}_T$ induces an isomorphism
 $Z^{(l+1)}_T\risom Z^{(l)}_T$ for $l=j,\dotsc,i-1$.  Hence
 $\vf^{(j)}_T\dotsb\vf^{(i)}_T$ induces an isomorphism $Z^{(i)}_T\risom
 T^{(j-1)}$.

 Finally, suppose $Z^{(i)}_T$ meets $e^{(l,i)}_TE^{(l)}_T$ for $l\neq
 j,k$, and let's find a contradiction.  If $l<j$, then interchange $l$
 and $j$.  Then, by the above, $T^{(j-1)}$ lies in both
 $e^{(k,\,j-1)}_TE^{(k)}_T$ and $e^{(l,\,j-1)}_TE^{(l)}_T$.  Therefore,
 $T^{(j-1)}$ is equal to their intersection, because $T^{(j-1)}$ is flat
 and its fibers are equal to those of the intersection.  It follows that
 $e^{(k,\,j)}_TE^{(k)}_T$ and $e^{(l,\,j)}_TE^{(l)}_T$ are disjoint on $
 F^{(j)}$.  But both these subschemes contain the image of $Z^{(i)}_T$,
 which is nonempty.  We have a contradiction, as desired.  The proof is
 now complete.
  \end{proof}

 \begin{prp}\label{prp-assdiag}
 Let $(t_0,\dots,t_n)$ be a strict sequence of arbitrarily near $T$-points
 of $F/Y$.  Equip the abstract ordered set of $t_i$ with the relation of
 proximity of\/ {\rm Definition~\ref{dfn:prox}}.  Then this set becomes
 an ordered proximity structure.
  \end{prp}
 \begin{proof}
 Law (P1) holds trivially.  

 As to (P2), suppose $t_i$ is proximate to $t_j$ and to $t_k$ with $j>k$.
 Then $T^{(i)}$ lies in $e^{(k+1,i)}_TE^{(k+1)}_T \bigcap
 e^{(j+1,i)}_TE^{(j+1)}_T$.  So Lemma~\ref{lem-Z} implies $t_j$ is
 proximate to $t_k$.  Furthermore, the lemma implies the intersection
 meets no $e^{(l+1,i)}_TE^{(l+1)}_T$ for $l\neq j,k$.  So $t_i$ is
 proximate to no third vertex $t_l$.  Thus (P2) holds.

 As to (P3), suppose $t_i$ and $t_j$ are each proximate to both $t_k$ and
 $t_l$ where $i>j>k>l$.  Given $p>k$, set $Z^{(p)}:=
 e^{(l+1,p)}_TE^{(l+1)}_T \bigcap e^{(k+1,p)}_TE^{(k+1)}_T$.  Then
 $T^{(i)}\subseteq Z^{(i)}$.  Now, $Z^{(i)}$ is $T$-flat with
 reduced points as fibers by Lemma~\ref{lem-Z}.  But $T^{(i)}$ is a
 similar $T$-scheme.  Hence $T^{(i)}= Z^{(i)}$.  Similarly, $T^{(j)}:=
 Z^{(j)}$.

 Lemma~\ref{lem-Z} yields $\vf^{(j+1)}_T\dotsb\vf^{(i)}_Te^{(m,i)}_T
 =e^{(m,j)}_T$ for $m=k,l$.  So $\vf^{(j+1)}_T\dotsb\vf^{(i)}_T$ carries
 $T^{(i)}$ into $T^{(j)}$.  Now, this map is proper, and both $T^{(i)}$
 and $T^{(j)}$ are $T$-flat with reduced points as fibers; hence, $T^{(i)}
 \risom T^{(j)}$.  It follows that
  $$
         \vf^{(j+2)}_T\dotsb\vf^{(i)}_TT^{(i)} \subseteq
   Z^{(j+1)}\subset \bigl(\vf^{(j+1)}_T\bigr)^{-1}T^{(j)}=E^{(j+1)}_T.
  $$
 Hence $Z^{(j+1)}$ meets $E^{(j+1)}_T$, contrary to Lemma~\ref{lem-Z}.
 Thus (P3) holds.
  \end{proof}

 \begin{dfn}\label{dfn:assdiag}
 Let's say that a strict sequence of arbitrarily near $T$-points of $F/Y$
 has {\it diagram} $(\I U,\,\theta)$ if $(\I U,\,\theta)$ is isomorphic
 to the ordered unweighted Enriques diagram coming from
 Propositions~\ref{prp-assdiag} and \ref{prp-ps}.
  \end{dfn}

 The following result was inspired by Ro\'e's Proposition 2.6 in
 \cite{Roe01}.

 \begin{thm}\label{thm:3-2} Fix an ordered unweighted Enriques diagram
 $(\I U,\,\theta)$ on $n+1$ vertices.  Then the strict sequences of
 arbitrarily near $T$-points of $F/Y$ with diagram $(\I U,\,\theta)$ form
 a functor; it is representable by a subscheme $F(\I U,\,\theta)$
 of $F^{(n)}$, which is $Y$-smooth with irreducible geometric fibers of
 dimension $\dim(\I U)$.
  \end{thm}
  \begin{proof}
  If a strict sequence of arbitrarily near $T$-points has diagram $(\I
 U,\,\theta)$, then, for any map $T'\to T$, the induced sequence of
 arbitrarily near $T'$-points plainly also has diagram $(\I U,\,\theta)$.
 So the sequences with diagram $(\I U,\,\theta)$ form a subfunctor of the
 functor of all sequences, which is representable by $F^{(n)}/Y$ by
 Proposition~\ref{prp-Har}.

 Suppose $n=0$.  Then $\I U$ has just one vertex.  So the two functors
 coincide, and both are representable by $F$, which is $Y$-smooth with
 irreducible geometric fibers of dimension 2.  However, $2=\dim(\I U)$.
 Thus the theorem holds when $n=0$.

 Suppose $n\ge1$.  Set $L:=\theta^{-1}n$.  Then $L$ is a leaf.  Set $\I
 T:=\I U-L$.  Then $\I T$ inherits the structure of an unweighted
 Enriques diagram, and it is ordered by the restriction $\theta|\I T$.
 By induction on $n$, assume the theorem holds for $(\I T,\, \theta|\I T)$.

   Set $G:=F(\I T,\, \theta|\I T)\subset F^{(n-1)}$ and
 $H:=\pi_n^{-1}G\subset F^{(n)}$.  Then $H$ represents the functor of
 sequences $(t_0,\dots,t_n)$ of arbitrarily near $T$-points such that
 $(t_0,\dots,t_{n-1})$ has diagram $(\I T,\, \theta|\I T)$ since
 $\pi^{(i)}\tau_{i}=\tau_{i-1}$ by Proposition~\ref{prp-Har}.  Moreover,
 $H$ is $G$-smooth with irreducible geometric fibers of dimension 2 by
 Lemma~\ref{lem:DerFam}.  And $G$ is $Y$-smooth with irreducible
 geometric fibers of dimension $\dim(\I T)$ as the theorem holds for $(\I
 T,\, \theta|\I T)$.  Thus $H$ is $Y$-smooth with irreducible geometric
 fibers of dimension $\dim(\I T)+2$.

 Let $(h_0,\dotsc,h_n)$ be the universal sequence of arbitrarily near
 $H$-points, and $H^{(i)}\subset F^{(i)}_H$ the image of $h_i$.  We must
 prove that $H$ has a largest subscheme $S$ over which
 $(h_0,\,\dots,\,h_n)$ restricts to a sequence with diagram $(\I
 U,\,\theta)$; we must also prove that $S$ is $Y$-smooth with irreducible
 geometric fibers of dimension $\dim(\I U)$.

 But, $(h_0,\,\dots,\,h_{n-1})$ has diagram $(\I T,\,\theta|\I T)$.  So
 $H^{(i)}$ satisfies the $i$ conditions of Definition~\ref{dfn:sseq} for
 $i=0,\dotsc, n-1$.  Hence $S$ is defined simply by the $n$ conditions on
 $H^{(n)}$: the $j$th requires $e^{(j,n)}_HE^{(j)}_H$ either (a) to be
 disjoint from $H^{(n)}$ or (b) to contain it as a subscheme; (b) applies
 if $L$ is proximate to $\theta^{-1}(j-1)$, and (a) if not, according to
 Definition~\ref{dfn:prox}.  Let $P$ be the set of $j$ for which (b)
 applies.  Set
  $$S:=h_n^{-1}\biggl(\bigcap_{j\in P}e^{(j,n)}_HE^{(j)}_H
              -\bigcup_{j\notin P}e^{(j,n)}_HE^{(j)}_H\biggr)$$
 Plainly, $S$ is  the desired largest subscheme of $H$.

 It remains to analyze the geometry of $S$.  First of all,
 $F^{(n)}_G=F^{(n)}\x_{F^{(n-1)}}G$ by Proposition~\ref{prp-Har}; so
 $F^{(n)}_G=H$ since $H:=\pi_n^{-1}G$.  Also,
 $F^{(n)}_H=F^{(n)}\x_{F^{(n-1)}}H$ and $h_n=(\zeta_n,1)$ where
 $\zeta_n\:H\into F$, again by Proposition~\ref{prp-Har}.  Hence
  $$F^{(n)}_H=F^{(n)}_G\x_GH=H\x_GH \text{ and } h_n=(1,1).$$
 Plainly, forming $e^{(j,n)}_T$ is functorial in $T$; whence,
 $e^{(j,n)}_HE^{(j)}_H = (e^{(j,n)}_GE^{(j)}_G)\x_GH$.  Hence,
 $h_n^{-1}e^{(j,n)}_HE^{(j)}_H=e^{(j,n)}_GE^{(j)}_G$.  Therefore,
  $$S:=\bigcap_{j\in P}e^{(j,n)}_GE^{(j)}_G -\bigcup_{j\notin
    P}e^{(j,n)}_GE^{(j)}_G.$$

 There are three cases to analyze, depending on $\type(L)$.  In any case,
         $$\dim(\I T)+\type(L) = \dim(\I U)$$
 owing to Formula~\ref{eq:2-1}.  Furthermore, each $e^{(j,n)}_G$ is an
 embedding.  So $e^{(j,n)}_GE^{(j)}_G$ has the form $\bb P(\Omega)$ for
 some locally free sheaf $\Omega$ of rank 2 on $G$ by
 Lemma~\ref{lem:DerFam} and Proposition~\ref{prp-Har}.  Hence
 $e^{(j,n)}_GE^{(j)}_G$ is $Y$-smooth with irreducible geometric fibers
 of dimension $\dim(\I T)+1$.

 Suppose $\type(L)=2$.  Then $L$ is a root.  So $P$ is empty, and by
 convention, the intersection $\bigcap_{j\in P}e^{(j,n)}_HE^{(j)}_H$ is
 all of $H$.  So
 $S$ is open in $H$, and maps onto $Y$.  Hence $S$ is $Y$-smooth with
 irreducible geometric fibers of dimension $\dim(H/Y)$, and
         $$\dim(H/Y)=\dim(\I T)+2=\dim(\I U).$$
 Thus the theorem holds in this case.

 Suppose $\type(L)=1$.  Then $L$ is a free vertex, but not a root.  So
 $L$ has an immediate predecessor, $M$ say.  Set $m:=\theta (M)$.  Then
 $P=\{m\}$.  So $S$ is open in $e^{(m,n)}_GE^{(m)}_G$, and maps onto $Y$.
 Hence $S$ is $Y$-smooth with irreducible geometric fibers of dimension
 $\dim(e^{(m,n)}_GE^{(m)}_G/Y)$, and
         $$\dim(e^{(m,n)}_GE^{(m)}_G/Y)=\dim(\I T)+1=\dim(\I U).$$
 Thus the theorem holds in this case too.

 Finally, suppose $\type(L)=0$.  Then $L$ is a satellite.  So $L$ is
 proximate to two vertices: an immediate predecessor, $M$ say, and a
 remote predecessor, $R$ say.  Set $m:=\theta (M)$ and $r:=\theta (R)$.
 Then $P=\{r,\,m\}$.  Set $Z:=e^{(r,n)}_GE^{(r)}_G\bigcap
 e^{(m,n)}_GE^{(m)}_G$.  Then $Z\risom G$ and $Z$ meets no
 $e^{(j,n)}_GE^{(m)}_G$ with $j\notin P$ owing to Lemma~\ref{lem-Z},
 because $(h_0,\,\dots,\,h_{n-1})$ is strict with diagram $(\I
 T,\,\theta|\I T)$.  Hence $S=Z$.  Therefore, $S$ is $Y$-smooth with
 irreducible geometric fibers of dimension $\dim(G/Y)$, and
         $$\dim(G/Y)=\dim(\I T)+0=\dim(\I U).$$
 Thus the theorem holds in this case too, and the proof is complete.
 \end{proof}

 \section{Isomorphism and enlargement}\label{scIae}
Fix a smooth family of geometrically irreducible surfaces $\pi\:F\to Y$.
In this section, we study the scheme $F(\I U,\,\theta)$ introduced in
Theorem~\ref{thm:3-2}.  First, we work out the effect of replacing the
ordering $\theta$ by another one $\theta'$.  Then we develop, in our
context, much of Ro\'e's Subsections 2.1--2.3 in \cite{Roe01};
specifically, we study a certain closed subset $E(\I U,\,\theta)\subset
F^{(n)}$ containing $F(\I U,\,\theta)$ set-theoretically.  Notably, we
prove that, if the sets $F(\I U',\,\theta')$ and $E(\I U,\,\theta)$
meet, then $E(\I U',\,\theta')$ lies in $E(\I U,\,\theta)$; furthermore,
$E(\I U',\,\theta')=E(\I U,\,\theta)$ if and only if $(\I U,\,\theta)
\cong (\I U',\,\theta')$.

Proposition~\ref{prpIso} below asserts that there is a natural
isomorphism $\Phi_{\theta,\theta'}$ from $F(\I U,\,\theta)$ to $F(\I
U,\,\theta')$.  On geometric points, $\Phi_{\theta,\theta'}$ is given as
follows.  A geometric point with field $K$ represents a sequence of
arbitrarily near $K$-points $(t_0,\dots,t_n)$ of $F/Y$.  To give $t_i$ is
the same as giving the local ring $A_i$ of the surface $F^{(i)}_K$ at
the $K$-point $T^{(i)}$, the image of $t_i$.  Set
$\alpha:=\theta'\circ\theta^{-1}$.  Then $\alpha i>\alpha j$ if $t_i$ is
proximate to $t_j$.  So there is a unique sequence $(\td t_0,\dots,\td
t_n)$ whose local rings $\wt A_j$ satisfy $A_i=\wt A_{\alpha i}$ in the
function field of $F_K$.  The sequences $(t_0,\dots,t_n)$ and $(\td
t_0,\dots,\td t_n)$ correspond under $\Phi_{\theta,\theta'}$.

To construct $\Phi_{\theta,\theta'}$, we must work with a sequence
$(t_0,\dots,t_n)$ of $T$-points for an arbitrary $T$.  To do so, instead
of the $A_i$, we use the transforms $e^{(i+1,n+1)}_TE^{(i+1)}_T$.  The
notation becomes more involved, and it is harder to construct $(\td
t_0,\dots,\td t_n)$.  We proceed by induction on $n$: we omit $t_n$,
apply induction, and ``reinsert'' $t_n$ as $\td t_{\alpha n}$.  Most of
the work is done in Lemma~\ref{lemAut}; the reinsertion is justified by
Lemma~\ref{lemL2}.

\begin{lem}\label{lemL2}
Let $(\td t_0,\dots,\td t_{n-1})$ be a strict sequence of arbitrarily
near $T$-points of $F/Y$, say with blowups $\wt F^{(i)}_T$ and so on.
Fix $l$, and let $T^{(l)}\subset \wt F^{(l)}_T$ be the image of a
section $t_l$ of $\wt F^{(l)}_T/T$.  Set $t_i:=\td t_i$ for $0\le i<l$,
and assume the sequence $(t_0,\dots,t_l)$ is strict.  Set
$T_{l}:=\wt T^{(l)}$ and
$T_{i}:=\wt\vf^{(l+1)}_T \dotsm \wt\vf^{(i)}_T\wt T^{(i)}$ for
$l < i < n$,
and assume $T^{(l)}$ and the $T_i$ are disjoint.  Then
$(t_0,\dots,t_l)$ extends uniquely to a strict sequence
$(t_0,\dots,t_n)$, say with blowups $F^{(i)}_T$ and so on, such that
 $t_l$ is a leaf and $F^{(l+1)}_T\x_{\smash{{F^{(l)}_T}}} \wt
F^{(i-1)}_T = F^{(i)}_T$ for $l< i\le n$.  Furthermore, the diagram of
$(t_0,\dots,t_n)$ induces that of $(\td t_0,\dots,\td t_{n-1})$.
\end{lem}
\begin{proof} Set $F_T^{(l)}:=\wt F_T^{(l)}$;
 let $F_T^{(l+1)}$ be the blowup of $F_T^{(l)}$ with center $T^{(l)}$,
and   $E_T^{(l+1)}$ be its exceptional divisor.  For $l< i\le n$, set
$F^{(i+1)}_T:=F^{(l+1)}_T\x_{F^{(l)}_T} \wt F^{(i)}_T$ and 
$T^{(i)}:=F^{(l+1)}_T\x_{F^{(l)}_T} \wt T^{(i-1)}$.
Now, $T^{(l)}$ and $T_i$ are disjoint for $l\le i < n$.  So
$F^{(i+1)}_T$ is the blowup of $F^{(i)}_T$ with center $T^{(i)}$.  Also,
$T^{(i)}$ is the image of a section $t_i$ of $F_T^{(i)}/T$.  Moreover,
since $(t_0,\dots,t_l)$ and $(\td t_0,\dots,\td t_{n-1})$ are strict
sequences, it follows that $(t_0,\dots,t_n)$ is a strict sequence too.
Furthermore, $t_l$ is a leaf, and the diagram of $(t_0,\dots,t_n)$
induces that of $(\td t_0,\dots,\td t_{n-1})$.  Plainly,
$(t_0,\dots,t_n)$ is unique.
\end{proof}

\begin{lem}\label{lemAut}
  Let $\alpha$ be a permutation of $\{0,\dots,n\}$.  Let
$(t_0,\dots,t_n)$ be a strict sequence of arbitrarily near $T$-points of
$F/Y$.  Assume that, if $t_i$ is proximate to $t_j$, then $\alpha
i>\alpha j$.  Then there is a unique strict sequence $(\td t_0,\dots,\td
t_n)$, say with blowups $\wt F^{(i)}_T$, exceptional divisors $\wt
E^{(i)}_T$, and so on, such that $F^{(n+1)}_T = \wt F^{(n+1)}_T$ and
$e^{(i,n+1)}_TE^{(i)}_T= \td e^{(\alpha' i,n+1)}_T\wt E^{(\alpha' i)}_T$
with $\alpha'i:=\alpha(i-1)+1$ for $1\le i\le n+1$; furthermore, $t_i$
is proximate to $t_j$ if and only if $\td t_{\alpha i}$ is proximate to
$\td t_{\alpha j}$.
\end{lem}
\begin{proof}
Assume $(\td t_0,\dots,\td t_n)$ exists.  
Let's prove, by induction on $j$, that both the sequence $(\td
t_0,\dots,\td t_j)$ and the map $F^{(n+1)}_T\to \wt F^{(j+1)}_T$ are
determined by the equality $F^{(n+1)}_T = \wt F^{(n+1)}_T$ and the $n+1$
equalities $e^{(i,n+1)}_TE^{(i)}_T= \td e^{(\alpha' i,n+1)}_T\wt
E^{(\alpha' i)}_T$ where $1\le i\le n+1$.
  If $j=-1$,
there's nothing to prove.  So suppose $j\ge0$.  Then $\wt T^{(j+1)}$ is
determined as the scheme-theoretic image of $\td e^{(j+2,n+1)}_T\wt
E^{(j+2)}_T$ by Lemma~\ref{lem-Z}.  So $\td t_{j+1}$ is determined.  But
then $\wt F^{(j+2)}_T$ is determined as the blowup of $\wt T^{(j+1)}$.
And $F^{(n+1)}_T\to \wt F^{(j+2)}_T$ is determined, because the preimage
of  $\wt T^{(j+1)}$ in  $F^{(n+1)}_T$ is a divisor.  Thus  $(\td
t_0,\dots,\td t_n)$  is unique.

To prove $(\td t_0,\dots,\td t_n)$ exists, let's proceed by induction on
$n$.  Assume $n=0$.  Then $\alpha=1$.  So plainly $\td t_0$ exists; just
take $\td t_0:=t_0$.

So assume $n\ge 1$.  Set $l:=\alpha n$.  Define a permutation $\beta$ of
$\{0,\dots,n-1\}$ by $\beta i:= \alpha i$ if $\alpha i<l$ and $\beta i:=
\alpha i-1$ if $\alpha i>l$.

Suppose $t_i$ is proximate to $t_j$ with $i<n$, and let us check that
$\beta i > \beta j$.  The hypothesis yields $\alpha i>\alpha j$.  So if
either $\alpha i<l$ or $\alpha j>l$, then $\beta i > \beta j$.  Now,
$\alpha i\neq l$ since $i<n$ and $l:=\alpha n$.  Similarly, $\alpha
j\neq l$ since $j<i$ as $t_i$ is proximate to $t_j$.  But if $\alpha
i>l$, then $\beta i:= \alpha i-1\ge l$, and if $\alpha j<l$, then $\beta
j:= \alpha j < l$.  Thus $\beta i > \beta j$.

Since $(t_0,\dots,t_{n-1})$ is strict, induction applies: there exists a
strict sequence $(\hat t_0,\dots,\hat t_{n-1})$, say with blowups $\wh
F^{(i)}_T$ and so forth, such that 
 $F^{(n)}_T = \wh F^{(n)}_T$ and $e^{(i,n)}_TE^{(i)}_T= \hat e^{(\beta'
i,n)}_T\wh E^{(\beta' i)}_T$ with $\beta'i:=\beta(i-1)+1$ for $1\le i\le
n$; furthermore, $t_i$ is proximate to $t_j$ if and only if $\td
t_{\beta i}$ is proximate to $\td t_{\beta j}$.  Set $\td t_i:=\hat t_i$
for $0\le i<l$.

Set $\td t_l:=\wh \vf^{(l+1)}_T \dotsm \wh \vf^{(n)}_Tt_n$ and $\wt
T^{(l)} := \wh \vf^{(l+1)}_T \dotsm\wh \vf^{(n)}_TT^{(n)}$.  Then $\td
t_l$ is a section of $\wh F^{(l)}_T/T$, and $\wt T^{(l)}$ is its image.
Note that, if $T^{(n)}$ meets $\hat e^{(j,n)}_T\wh E^{(j)}_T$ with $1\le
j\le n$, then $T^{(n)}$ is contained in $\hat e^{(j,n)}_T\wh E^{(j)}_T$,
because $\hat e^{(j,n)}_T\wh E^{(j)}_T=e^{(i,n)}_TE^{(i)}_T$ for
$i:=\beta'^{-1}j$ and because $(t_0,\dots,t_n)$ is strict.  Furthermore,
if so, then $l>j$, because $t_n$ is proximate to $t_i$, and so $\alpha
n>\alpha i$, or $l>\beta i=j$; moreover, then $\wt T^{(l)}$ is contained
in $\hat e^{(j,l)}_T\wh E^{(j)}_T$, because the latter is equal to $\wh
\vf^{(l+1)}_T \dotsm\wh \vf^{(n)}_T\hat e^{(j,n)}_T\wh E^{(j)}_T$ since
$l>j$.

Suppose $\wt T^{(l)}$ meets  $\hat e^{(k,l)}_T\wh E^{(k)}_T$.  Then
$T^{(n)}$ meets $(\wh \vf^{(l+1)}_T \dotsm \wh \vf^{(n)}_T)^{-1}\hat e^{(k,l)}_T\wh
E^{(k)}_T$.  So $T^{(n)}$ meets one of the latter's components, which is
a $\hat e^{(j,n)}_T\wh E^{(j)}_T$ for some $j$.  Hence $T^{(l)} \subset
\hat e^{(j,l)}_T\wh E^{(j)}_T$, as was noted above.  Now, the map $\hat
e^{(j,n)}_T\wh E^{(j)}_T\to \wh F^{(l)}_T$ factors through $\wh
E^{(k)}_T$, and its image is $\hat e^{(j,l)}_T\wh E^{(j)}_T$, as was
noted above.  So $\hat e^{(j,l)}_T\wh E^{(j)}_T$ is contained in $\hat
e^{(k,l)}_T\wh E^{(k)}_T$; whence, the two coincide, since they are flat
and coincide on the fibers over $T$.  Thus $\wt T^{(l)}$ is contained in
$\hat e^{(k,l)}_T\wh E^{(k)}_T$.  Hence, since $(\td t_0,\dots,\td
t_{l-1})$ is strict, so is $(\td t_0,\dots,\td t_l)$.  Furthermore,
$T^{(n)}$ is contained in $\hat e^{(k,n)}_T\wh E^{(k)}_T$.  Thus if $\td
t_l$ is proximate to $\td t_k$, then $t_n$ is proximate to $t_i$ for
$i:=\beta'^{-1}k$.  Moreover, the converse follows from what was noted
above.

Set $T_{l}:=\wt T^{(l)}$ and
$T_{i}:=\wt\vf^{(l+1)}_T \dotsm \wt\vf^{(i)}_T\wt T^{(i)}$ for
$l < i < n$.  Then $\wt T^{(l)}$ meets no $\wt T_{i}$, because,
otherwise, $T^{(n)}$ would meet $(\wh \vf^{(l+1)}_T \dotsm \wh
\vf^{(n)}_T)^{-1}\wt T_{i+1}$, and so $T^{(n)}$ would meet some $\hat
e^{(j,n)}_T\wh E^{(j)}_T$ with $l < j$, contrary to the note
above.  So Lemma~\ref{lemL2} implies  $(\td t_0,\dots, \td
t_l)$ extends to a strict sequence $(\td t_0,\dots, \td t_n)$ such that
 $\td t_l$ is a leaf and $\wt F^{(l+1)}_T\x_{\smash{{\wt F^{(l)}_T}}}
\wh F^{(i)}_T = \wt F^{(i+1)}_T$ for $l< i\le n$; furthermore, the
diagram of $(\td t_0,\dots,\td t_n)$ induces that of $(\hat
t_0,\dots,\hat t_{n-1})$.

Therefore, $t_i$ is proximate to $t_j$ if and only if $\td
t_{\alpha i}$ is proximate to $\td t_{\alpha j}$ for $0\le i\le n$,
because $t_i$ is proximate to $t_j$ if and only if $t_{\beta i}$ is
proximate to $t_{\beta j}$ for $0\le i < n$ and because $t_n$ is
proximate to $t_j$ if and only if $\td t_l$ is proximate to $\td t_k$
for $k:=\beta' j$.

Recall from above that $F^{(n)}_T = \wh F^{(n)}_T$ and $\wt
F^{(l+1)}_T\x _{\smash{{\wt F^{(l)}_T}}} \wh F^{(n)}_T = \wt F^{(n+1)}_T$.
But this product is equal to the blowup of $\wh F^{(n)}_T$ along $T^{(n)}$
since $\wt T^{(l)}$ meets no $\wh T_{i}$.  And the blowup of $F^{(n)}_T$
along $T^{(n)}$ is $F^{(n+1)}_T$.  Thus $F^{(n+1)}_T = \wt F^{(n+1)}_T$.

Recall $e^{(i,n)}_TE^{(i)}_T= \hat e^{(\beta' i,n)}_T\wh E^{(\beta'
i)}_T$ for $1\le i\le n$.  Hence, $e^{(i,n+1)}_TE^{(i)}_T$ is equal to
the image of a natural embedding of $\hat e^{(\beta' i,n)}_T\wh
E^{(\beta' i)}_T$ in $\wt F^{(n+1)}_T$.  In turn, this image is equal to
$\wt e^{(\alpha' i,n)}_T\wt E^{(\alpha' i)}_T$ since $\wt T^{(l)}$ meets
no $\wh T_{i}$.  Similarly, $E^{(n+1)}_T = \td e^{(l+1,n+1)}_T\wt
E^{(l+1)}_T$.  Thus $e^{(i,n+1)}_TE^{(i)}_T = \td e^{(\alpha'
i,n+1)}_T\wt E^{(\alpha' i)}_T$ for $1\le i\le n+1$.
\end{proof}

\begin{prp}\label{prpIso}
 Fix an unweighted Enriques diagram $\I U$.
Then, given two orderings $\theta$ and $\theta'$, there exists a
natural isomorphism
 $$\Phi_{\theta,\theta'}\:F(\I U,\,\theta)
 \risom F(\I U,\,\theta').$$
Furthermore,  $\Phi_{\theta,\theta}=1$, and $\Phi_{\theta',\theta''}\circ
 \Phi_{\theta,\theta'}=\Phi_{\theta,\theta''}$ for any third ordering
$\theta''$.
\end{prp}
\begin{proof}
Say $\I U$ has $n+1$ vertices.  Set $\alpha:=\theta'\circ\theta^{-1}$.
Then $\alpha$ is a permutation of $\{0,\dots,n\}$.

Each $T$-point of $F(\I U,\,\theta)$ corresponds to a strict
sequence $(t_0,\dots,t_n)$ owing to Theorem~\ref{thm:3-2}.  For each
$i$, say $t_i$ corresponds to the vertex $V_i$ of $\I U$.  Then
$\theta(V_i)=i$, and if $t_i$ is proximate to $t_j$, then $V_i$ is
proximate to $V_j$.  So $\theta'(V_i)>\theta'(V_j)$ since $\theta'$ is
an ordering.  Hence $\alpha i>\alpha j$.

Therefore, by Lemma~\ref{lemAut}, there is a unique strict sequence
$(\td t_0,\dots, \td t_n)$ such that $t_i$ is proximate to $t_j$ if and
only if $\td t_{\alpha i}$ is proximate to $\td t_{\alpha j}$.  Plainly
$(\td t_0,\dots, \td t_n)$ has $(\I U,\,\theta')$ as its diagram.  Hence
$(\td t_0,\dots, \td t_n)$ corresponds to a $T$-point of $F(\I
U,\,\theta')$ owing to Theorem~\ref{thm:3-2}.

Due to uniqueness, sending $(t_0,\dots,t_n)$ to $(\td t_0,\dots, \td
t_n)$ gives a well-defined map of functors.  It is represented by a map
$\Phi_{\theta,\theta'}\:F(\I U,\,\theta) \to F(\I U,\,\theta')$.  Again
due to uniqueness, $\Phi_{\theta,\theta}=1$ and
$\Phi_{\theta',\theta''}\circ \Phi_{\theta,\theta'} =
\Phi_{\theta,\theta''}$ for any $\theta''$.  So $\Phi_{\theta',\theta}
\circ \Phi_{\theta,\theta'} = 1$ and $\Phi_{\theta,\theta'} \circ
\Phi_{\theta',\theta} = 1$.  Thus $\Phi_{\theta,\theta'}$ is an
isomorphism, and the proposition is proved.
\end{proof}

\begin{cor}\label{corAut}
 Fix an ordered unweighted Enriques diagram $(\I U,\,\theta)$.  Then
there is a natural free right action of $\Aut(\I U)$ on $F(\I
U,\,\theta)$; namely, $\gamma\in \Aut(\I U)$ acts as
$\Phi_{\theta,\theta'}$ where $\theta':=\theta\circ\gamma$.
\end{cor}
\begin{proof}
Let $V\in\I U$ be a vertex that precedes another $W$.  Then $\gamma(V)$
precedes $\gamma(W)$ because $\gamma\in \Aut(\I U)$.  Since $\theta$ is
an ordering, $\theta(\gamma(V)) \le \theta(\gamma(W))$.  Hence $\theta'(V)
\le \theta'(W)$.  Thus $\theta'$ is an ordering.

So there is a natural isomorphism $\Phi_{\theta,\theta'}\:F(\I
U,\,\theta) \risom F(\I U,\,\theta')$ by Proposition~\ref{prpIso}.
Now, $\gamma$ induces an isomorphism of ordered unweighted Enriques
diagrams from $(\I U,\,\theta')$ to $(\I U,\,\theta)$; hence,
$F(\I U,\,\theta')$ and $F(\I U,\,\theta)$ are the same
subscheme of $F^{(n)}$, and $\Phi_{\theta,\theta'}$ is an automorphism
of $F(\I U,\,\theta)$.

Note that, if $\gamma=1$, then $\theta'=\theta$; moreover,
$\Phi_{\theta,\theta}=1$.

Given $\delta\in \Aut(\I U)$, set $\theta'':=\theta'\circ\delta$ and
$\theta^*:=\theta\circ\delta$.  Then $\gamma$ also induces an isomorphism
from $(\I U,\,\theta'')$ to $(\I U,\,\theta^*)$, and so
$\Phi_{\theta',\theta''}$ and $\Phi_{\theta,\theta^*}$ coincide.  Now,
$\Phi_{\theta',\theta''}\circ \Phi_{\theta,\theta'} =
\Phi_{\theta,\theta''}$.  Thus $\Aut(\I U)$ acts on
$F(\I U,\,\theta)$, but it acts on the right because $\theta''$ is
equal to $\theta\circ(\gamma\delta)$, not to  $\theta\circ(\delta\gamma)$.

Suppose $\gamma$ has a fixed $T$-point.  Then the $T$-point is fixed
under $\Phi_{\theta,\theta'}$.  Now, we defined $\Phi_{\theta,\theta'}$
by applying Lemma~\ref{lemAut} with $\alpha:=\theta'\circ\theta^{-1}$.
And the lemma asserts that $\alpha$ is determined by its action on the
$e^{(i,n+1)}_TE^{(i)}_T$.  But this action is trivial because the
$T$-point is fixed.  Hence $\alpha=1$.  But
$\alpha=\theta\circ\gamma\circ\theta^{-1}$.  Therefore, $\gamma=1$.
Thus the action of $\Aut(\I U)$ is free, and the corollary is proved.
\end{proof}

\begin{cor}\label{corQt}
  Fix an ordered unweighted Enriques diagram $(\I U,\,\theta)$, and let
$G\subset{\Aut}(\I U)$ be a subgroup.  Then the quotient $F(\I
U,\,\theta)\big/G$ is $Y$-smooth with irreducible geometric fibers of
dimension $\dim(\I U)$.
\end{cor}
\begin{proof}
The action of $G$ on $F(\I U,\,\theta)$ is free by
Corollary~\ref{corAut}.  So $G$ defines a finite flat equivalence
relation on $F(\I U,\,\theta)$.  Therefore, the quotient exists, and the
map $F(\I U,\,\theta)\to F(\I U,\,\theta)\big/G$ is faithfully flat.  Now,
$F(\I U,\,\theta)$ is $Y$-smooth with irreducible geometric fibers of
dimension $\dim(\I U)$ by Theorem~\ref{thm:3-2}; so 
$F(\I U,\,\theta)\big/G$ is too.
\end{proof}

\begin{dfn}\label{dfnEff}
For $1\le i\le j$, set $E^{(i,i)}:=E^{(i)}$ and 
	$$E^{(i,j)}:=(\vf^{(i+1)}\dotsm\vf^{(j)})^{-1}E^{(i)}
 \text{ if }i<j.$$

Given an ordered unweighted Enriques diagram $(\I U,\,\theta)$ on $n+1$
vertices, say with proximity matrix $(p_{ij})$, let $E(\I U,\,\theta)
\subset F^{(n)}$
be the set of scheme points $\I t$ such that, on the fiber
$F^{(n+1)}_{\I t}$, for $1\le k\le n$, the divisors $\sum_{i=k}^{n+1}
p_{ik}  E^{(i,n+1)}_{\I t}$ are effective.
\end{dfn}

\begin{prp}\label{prpClsd}
Let $(\I U,\,\theta)$ be an ordered unweighted Enriques diagram.  Then
$E(\I U,\,\theta)$ is closed and contains $F(\I U,\,\theta)$
set-theoretically.
\end{prp}
\begin{proof}
Say $\I U$ has $n+1$ vertices.  Fix $\I t\in F^{(n)}$ and $1\le k\le n$.
If $\I t\in F(\I U,\,\theta)$, then, as is easy to see by induction on
$j$ for $k\le j\le n$, the divisor $\sum_{i=k}^{j+1} p_{ik}
E^{(i,j+1)}_{\I t}$ is equal to the 
strict transform on $F^{(j+1)}_{\I
t}$ of $E^{(k)}_{\I t}$, in other words, to $e^{(k,j+1)}_TE^{(k)}_T$
where $T:=\Spec \kappa(\I t)$.  Hence $E(\I U,\,\theta)$ contains $F(\I
U,\,\theta)$.

Set $\wt E^{(k)}:=\sum_{i=k}^{n+1}p_{ik}E^{(i,n+1)}$.  Then $\tu
h^0\bigl(F^{(n+1)}_{\I t},\,\mc O(\wt E^{(k)}_{\I t})\bigr)\le 1$ for
any $\I t$, and equality holds if and only if $\I t\in F(\I
U,\,\theta)$, as the following essentially standard argument shows.
Plainly, it suffices to show that, if $\wt E^{(k)}_{\I t}$ is linearly
equivalent to an effective divisor $D$, then $\wt E^{(k)}_{\I t} = D$.

  Let $H$ be the preimage on $F^{(n+1)}_{\I t}$ of an ample divisor on
$F_{\I t}$.  Then the intersection number $\wt E^{(k)}_{\I t}\cdot H$
vanishes by the projection formula because each component of $\wt
E^{(k)}_{\I t}$ maps to a point in $F_{\I t}$.  So $D \cdot H$ vanishes
too.  Hence each component of $D$ must also map to a point in $F_{\I t}$
because $D$ is effective and $H$ is ample.  Hence $D$ is some linear
combination of the $E^{(i,n+1)}_{\I t}$ because they form a basis of the
group of divisors whose components each map to a point.  Furthermore,
the combining coefficients must be the $ p_{ik}$ because these
coefficients are given by the intersection numbers with the
$E^{(i,n+1)}_{\I t}$.  Thus $\wt E^{(k)}_{\I t} = D$.

Thus $E(\I U,\,\theta)$ is the set of $\I t\in F^{(n)}$ such that $\tu
h^0\bigl(F^{(n+1)}_{\I t},\,\mc O(\wt E^{(k)}_{\I t})\bigr)\ge 1$ for
all $k$.  Hence $E(\I U,\,\theta)$ is closed by semi-continuity
\cite[Thm.~(7.7.5), p.~67]{EGAIII2}.
\end{proof}

\begin{prp}\label{prpCtd}
Let $(\I U,\,\theta)$ and $(\I U',\,\theta')$ be two ordered unweighted
Enriques diagrams on $n+1$ vertices, and let $\I P$ and $\I P'$ be their
proximity matrices.  Then the following conditions are equivalent:
\begin{enumerate}
 \item The sets $F(\I U',\,\theta')$ and  $E(\I U,\,\theta)$ meet.
 \item The set $E(\I U',\,\theta')$ is contained in the set $E(\I
U,\,\theta)$.
 \item The matrix  $\I P'^{-1}\I P$ only has nonnegative entries.
\end{enumerate}
Furthermore, $E(\I U',\,\theta') = E(\I U,\,\theta)$ if and only if $(\I
U,\,\theta) \cong (\I U',\,\theta')$.
\end{prp}
\begin{proof}
Fix $\I t\in F^{(n)}$, and define two sequences 
of divisors on
$F^{(n+1)}_{\I t}$ by these matrix equations:
\begin{align*}
	(\wt E^{(1)}_{\I t},\dotsc,\wt E^{(n+1)}_{\I t})
	 &= (E^{(1,n+1)}_{\I t},\dotsc, E^{(n+1,n+1)}_{\I t})\I P;\\
	(\wt E^{(1)}_{\I t}{}',\dotsc,\wt E^{(n+1)}_{\I t}{}')
     &= (E^{(1,n+1)}_{\I t},\dotsc, E^{(n+1,n+1)}_{\I t})\I P'.
\end{align*}
These two equations imply the following one:
\begin{equation}\label{equationCtd1}
	(\wt E^{(1)}_{\I t},\dotsc,\wt E^{(n+1)}_{\I t})
= (\wt E^{(1)}_{\I t}{}',\dotsc,\wt E^{(n+1)}_{\I t}{}')\I P'^{-1}\I P;
\end{equation}
in other words, $\wt E^{(j)}_{\I t}=\sum_{k=1}^{n+1}q_{kj}\wt
E^{(k)}_{\I t}{}'$ where say $(q_{kj}):=\I P'^{-1}\I P$.

Suppose $\I t\in F(\I U',\,\theta')$.  Then $\wt E^{(k)}_{\I t}{}'$ is
the proper transform on $F^{(n+1)}_{\I t}$ of $E^{(k)}_{\I t}$, as we
noted at the beginning of the proof of Proposition~\ref{prpClsd}.  So
the $\wt E^{(k)}_{\I t}{}'$ form a basis of the group of divisors whose
components each map to a point in $F_{\I t}$.  Hence,
by~(\ref{equationCtd1}), if $\wt E^{(j)}_{\I t}$ is effective, then
$q_{kj}\ge0$ for all $k$.  Thus (1) implies (3).

Suppose $\I t\in E(\I U',\,\theta')$.  Then $\wt E^{(k)}_{\I t}{}'$ is
effective.  Suppose too $q_{kj}\ge0$ for all $k, j$.  Then $\wt
E^{(j)}_{\I t}$ is effective for all $j$ by (\ref{equationCtd1}).  So
$\I t\in E(\I U,\,\theta)$.  Thus (3) implies (2).

By Proposition~\ref{prpClsd}, $E(\I U,\,\theta)$ contains $F(\I
U,\,\theta)$.  By Theorem~\ref{thm:3-2}, $F(\I U,\,\theta)$ is nonempty.
Thus (2) implies (1).  So (1), (2), and (3) are equivalent.

Furthermore, suppose $E(\I U',\,\theta') = E(\I U,\,\theta)$.  Then both
$\I P'^{-1}\I P$ and $\I P^{-1}\I P'$ have nonnegative entries since (2)
implies (3).  But each matrix is the inverse of the other, and both are
lower triangular.  Hence both are the identity.  So $\I P' =\I P$;
whence, $(\I U,\,\theta) \cong (\I U',\,\theta')$.  The converse is
obvious.  Thus the proposition is proved.
\end{proof}

\section{The Hilbert scheme}\label{sc:Hilb}
Fix a smooth family of geometrically irreducible surfaces $\pi\:F\to Y$.
In this section, we prove our main result, Theorem~\ref{thmF2H}.  It
asserts that, given an Enriques diagram $\I D$ and an ordering $\theta$,
there exists a natural map $\Psi$ from the quotient $F(\I D,\,\theta)
\big/{\Aut}(\I D)$ into the Hilbert scheme $\Hilb^d_{F/Y}$ with
$d:=\deg\I D$ and with $F(\I D,\,\theta):=F(\I U,\,\theta)$ where $\I U$
is the unweighted diagram underlying $\I D$.

The quotient $F(\I D,\,\theta) \big/{\Aut}(\I D)$ parameterizes the
strict sequences of arbitrarily near points of $F/Y$ with diagram $(\I
U,\,\theta)$, up to automorphism of $\I D$.  The image of $\Psi$
parameterizes the (geometrically) complete ideals of $F/Y$ with diagram
$\I D$.  The map $\Psi$ is universally injective.  In fact, $\Psi$ is an
embedding in characteristic $0$.  However, in positive characteristic,
$\Psi$ can be purely inseparable; Appendix B discusses  examples found
by Tyomkin. 

 We close this section with Proposition~\ref{prpM}, which addresses the
 important special case where every vertex of $\I D$ is a root; here,
 $\Psi$ is an embedding in any characteristic.  Further, other
 examples in Appendix B show that $\Psi$ can remain an embedding even
 after a nonroot is added.

\begin{sbs}[Geometrically complete ideals] \label{sb:ci}
 Let $K$ be a field, $(t_0,\dotsc,t_n)$ a sequence of arbitrarily near
$K$-points of $F/Y$.  Since $\Spec(K)$ consists of a single reduced
point, the sequence is strict.  Let $(\I U, \theta)$ be its diagram in
the sense of Definition~\ref{dfn:assdiag}.

Suppose $\I U$ underlies an Enriques diagram $\I D$, say with weights
$m_V$ for $V\in \I U$.  Using the divisors $E^{(i,n+1)}_K$ on
$F^{(n+1)}_K$ of Definition~\ref{dfnEff}, set
	$$E_K:=\tsum_Vm_VE^{(\theta(V)+1,\,n+1)}_K \text{ and }
	 \mc L_K:=\mc O_{F^{(n+1)}_K}(-E_K).$$
Given $V\in \I U$, set $j:=\theta(V)$ and $D_V:=e^{(j+1,n+1)}_KE^{(j+1)}_K$.
Inspired by Lipman's remark  \cite[p.~306]{Li94},
let's compute the intersection number $-(E_K\cdot D_V)$, that is,
$\deg(\mc L|D_V)$.  Plainly, $(E^{(j+1,\,n+1)}_K\cdot D_V)=-1$.  And, for
  $W\neq V$, plainly $(E^{(\theta(W)+1,\,n+1)}_K\cdot D_V)$ is equal to $1$ if
$W\succ V$, and to $0$ if not.  Hence $-(E_K\cdot D_V)$ is equal to
$m_V-\tsum_{W\succ V}m_W$, which is at least $0$  by the Proximity
Inequality. 

Set $\vf_K:=\vf^{(1)}_K\dotsb\vf^{(n+1)}_K$, and form $\mc
I:=\vf_{K*}\mc L_K$ on $F_K$.  Then $\mc I$ is a complete ideal, one
that is integrally closed; also, $\mc I\mc O_{F^{(n+1)}_K}=\mc L_K$ and
$R^q\vf_{K*}\mc L_K =0$ for $q\ge1$.  These three statements hold since
$(E_K\cdot D_V)\le0$ for all $V$ and, as is well known, $R^q\vf_{K*}\mc
O_{F^{(n+1)}_K} = 0$ for $q\ge1$; see Lipman's discussion \cite[\S18,
p.~238]{Li69} and his Part~(ii) of \cite[Thm.~(12.1), p.~220]{Li69};
also see Deligne's Th\'eor\`eme 2.13 \cite[p.~22]{D73}.  Furthermore,
	$$\dim_K\tu H^0(\mc O_{F_K}/\mc I) = d
 \text{\quad where } d:=\deg \I D.$$
This formula is a modern version of Enriques' formula \cite[Vol.~II,
p.~426]{EC15}; it was proved in different ways independently by Hoskin
\cite[5.2, p.~85]{Ho56}, Deligne \cite[2.13, p.~22]{D73}, and Casas
\cite[6.1, p.~438]{Ca90}; Hoskin and Deligne worked in greater
generality, Casas worked over $\bb C$.

The $m_V$ are determined by $\mc I$ because the divisors $E^{(i,n+1)}_K$
are numerically independent; their intersection numbers with divisors
are defined because they are complete.  The $m_V$ may be found as
follows.  Let $\mc P$ be the ideal of the image $T^{(0)}$ of $t_0$,
which is a $K$-point of $F_K$.  Let $m$ be the largest integer such that
$\mc P^m\supset \mc I$.  Then $m =m_V$ where $V:=\theta^{-1}(0)$, since
$\mc P\mc O_{F^{(n+1)}_K}=\mc O_{F^{(n+1)}_K}(-E^{(1, n+1)}_K)$.  Note
in passing that $\mc P$ is a minimal prime of $\mc I$ since $m_V\ge1$.

The remaining $m_W$ can be found by recursion.  Indeed, on $F^{(1)}_K$,
form the ideal $\mc I':=\mc I\mc O(m_VE^{(1)})$.  Then $\mc I'$ is the
direct image from $F^{(n+1)}_K$ of $\mc O(-E_K')$ where $E_K':=
  \tsum_{W\neq V}m_WE^{(\theta(W)+1,\,n+1)}_K$.  Hence $\mc I'$ is the
complete ideal associated to the sequence $(t_1,\dotsc,t_n)$ of
arbitrarily near $K$-points of $F^{(1)}/Y$ and to the ordered Enriques
diagram $(\I D', \theta')$ where $\I D':=\I D-V$ and
$\theta'(W):=\theta(W)-1$.
  
The ideal $\mc I$ determines the diagram $\I D$.  Indeed, for $0\le i\le
n$, let $A_i,\,\I m_i $ be the local ring of the surface $F^{(i)}_K$ at
the $K$-point that is the image of $t_i$.  Then according to Lipman's
preliminary discussion in \cite[p.~294--295]{Li94}, the set $\{A_i\}$
consists precisely of 2-dimensional regular local $K$-domains whose
fraction field is that of $F_K$ and whose maximal ideal contains the
stalk of $\mc I$ at some point of $F_K$.  Furthermore, $t_i$ is
proximate to $t_j$ if and only if $A_i$ is contained in the ring of the
valuation $v_j$ defined by the formula: $v_j(f):=\max\{m\mid f\in \I
m_j^m\}$.  Finally, if $W:=\theta^{-1}(j)$, then the weight $m_W$ is the
largest integer $m$ such that $\I m_j^m$ contains the appropriate stalk
of $\mc I$.

Let $\mc J$ be an arbitrary ideal on $F_K$ of finite colength.  Let
$L/K$ be an arbitrary field extension.  If the extended ideal $\mc J_L$
on $F_L$ is complete, then $\mc J$ is complete, and the converse holds
if $L/K$ is separable; see Nobile and Villamayor's proof of
\cite[Prp.~(3.2), p.~251]{NV97}.   Let us say that $\mc J$ is {\it
geometrically complete\/} if $\mc J_L$ on $F_L$ is complete for every
$L$, or equivalently, for some algebraically closed $L$.  In
characteristic $0$, if $\mc J$ is complete, then it is geometrically
complete.

The extended ideal $\mc I_L$ on $F_L$ is, plainly, the complete ideal
associated to the extension of the sequence $(t_0,\dotsc,t_n)$ and to
the same ordered Enriques diagram $(\I D, \theta)$.  Hence $\mc I$ is
geometrically complete.

Suppose that $K$ is algebraically closed.  Suppose that $\mc J$ is
complete and that $\dim_K\tu H^0(\mc O_{F_K}/\mc J)$ is finite and
nonzero.  Then $\mc J$ arises from some sequence $(s_0,\dotsc,s_n)$ and
some ordered Enriques diagram.  Indeed, choose a minimal
prime $\mc P$ of $\mc J$.  Then $K\risom\mc O_{F_K}/\mc P$ since $K$ is
algebraically closed.  Hence $\mc P$ defines a $K$-point $S^{(0)}$ of
$F_K$, so a section $s_0$ of $F_K/K$.  Set $m_0:=\max\{\,m\mid\mc
P^m\supset \mc J\,\}$.

Let $F'_K$ be the blowup of $F_K$ at $S^{(0)}$, and $E'_K$ the
exceptional divisor.  Set $\mc J':=\mc J\mc O_{F'_K}(m_0E'_K)$.  Then
$\mc J'$ is complete by Zariski and Samuel's \cite[Prp.~5, p.~381]{ZS}.
If $\mc J'=\mc O_{F'_K}$, then stop.  If not, then repeat the process
again and again, obtaining a sequence $(s_0,s_1,\dotsc)$.  Only finitely
many repetitions are necessary because, as Lipman \cite[p.~295]{Li94}
points out, the local ring of $F^{(i)}_K$ at $S^{(i)}$ is dominated by a
Rees valuation of $\mc J$, that is, the valuation associated to an
exceptional divisor of the normalized blowup of $\mc J$.  Then $\mc J'$
arises from the sequence of $s_i$ weighted by the $m_{\theta^{-1}(i)}$
owing to Lipman's \cite[prp.~(6.2), p.~208]{Li69} and discussion
before it.
\end{sbs}

\begin{lem}\label{lemDVR1}
 Let $A$ be a discrete valuation ring, set $T:=\Spec A$, and denote by
$\eta\in T$ the generic point and by $y\in T$ the closed point.  Fix a
map $T\to Y$.  Let $\I D$ be an Enriques diagram, say with $n+1$
vertices, and $\mc I$ a coherent ideal on $F_T$ that generates
geometrically complete ideals on $F_\eta$ and $F_y$, each with diagram
$\I D$.  Let $\theta$ be an ordering of $\I D$, and $\tilde{\I t}$ a
$k(\eta)$-point of $F(\I D,\,\theta)$ such that $\mc I_\eta$ generates
an invertible sheaf on  $F^{(n+1)}_\eta$.  Then
$\tilde{\I t}$ extends to a $T$-point $\I t$ of $F(\I D,\,\theta)$.
\end{lem}

\begin{proof}
Let $\theta'$ be a second ordering.  By the construction of the
isomorphism $\Phi_{\theta,\theta'}$ in the proof of
Proposition~\ref{prpIso}, a $T$-point of $F(\I D,\,\theta)$ corresponds
to the $T$-point of $F(\I D,\,\theta')$ given by Lemma~\ref{lemAut} with
$\alpha:=\theta'\circ\theta^{-1}$.  Moreover, the lemma says that
$F_T^{(n+1)}$ is unchanged.  It follows that, to construct $\I t$, we
may replace $\theta$ by $\theta'$.  Thus we may assume that $E(\I
D,\,\theta)$ is a minimal element among the various closed subsets $E(\I
D,\,\theta')$ of $F^{(n)}$.

Let $R\in \I D$ be a root, and temporarily set $i:=\theta(R)$.  Say
$\tilde {\I t}$ corresponds to the sequence of blowups
$F^{(j+1)}_\eta\to F^{(j)}_\eta$ with centers $\eta_j$.  The image of
$\eta_i$ in $F_T$ is a $k(\eta)$-point; denote its closure by $T_R$.
Since $A$ is a discrete valuation ring, the structure map is an
isomorphism $T_R\risom T$.

Let $Z\subset F_T$ be the subscheme with ideal $\mc I$.  Its fibers
$Z_\eta$ and $Z_y$ are finite, and both have degree $\deg(\I D)$ since
the two ideals are geometrically complete with diagram $\I D$ by
hypothesis.  Since $T$ is reduced, $Z$ is $T$-flat.

As $R$ varies, the points $(T_R)_\eta$ are exactly the components of
$Z_\eta$ again because its ideal $\mc I_\eta$ is geometrically complete
with diagram $\I D$.  Hence the several $T_R$ are just the components of
$Z$ that meet $Z_\eta$.  But every component of $Z$ meets $Z_\eta$ since
$Z$ is $T$-flat.  Thus the $T_R$ are the the components of $Z$.

Since $T_R\risom T$ for each $R$, the fiber $(T_R)_y$ is a single point,
so a component of the discrete set $Z_y$.  The number of $T_R$ is the
number of roots of $\I D$, which is also the number of points of $Z_y$.
Hence the several $T_R$ are disjoint.

 Given $R$, let $m_R$ be its weight, $\mc P_R$ the ideal of $T_R$ in
$F_T$. Then $(\mc P_R^{m_R})_\eta\supset \mc I_\eta$.  Let's see that
$\mc P_R^{m_R}\supset \mc I$.  Indeed, form the image, $\mc M$ say, of
$\mc I$ in $\mc O_{F_T}\big/\mc P_R^{m_R}$.  Then $\mc M_\eta=0$.  Let
$u\in A$ be a uniformizing parameter.  Then $\mc M$ is annihilated by a
power of $u$.  Now, $\mc P_R$ is quasi-regular by \cite[(17.12.3),
p.~83]{EGAIV4} since $T_R\risom T$ and $F_T$ is $T$-smooth.  Hence $\mc
P_R^j/P_R^{j+1}$ is $T$-flat for all $j$ by \cite[(16.9.4),
p.~47]{EGAIV4}.  Hence $\mc O_{F_T}\big/\mc P_R^{m_R}$ is $T$-flat.  So
$u$ is a nonzerodivisor on $\mc O_{F_T}\big/\mc P_R^{m_R}$.  Hence $\mc
M=0$.  Thus $\mc P_R^{m_R}\supset \mc I$.

Let $n_R$ be the largest integer such that $(\mc P_R^{n_R})_y\supset \mc
I_y$.  Then $n_R\ge m_R$.  Now, $\mc I_y$ is geometrically complete with
diagram $\I D$.  Hence $n_R$ is the weight of the root corresponding to
$(\mc P_R)_y$.  Hence $\sum_Rn_R=\sum_Rm_R$.  But $n_R\ge m_R$.
Therefore, $n_R= m_R$ for every root $R$.

Let $\I D'$ be the diagram obtained from $\I D$ by omitting the roots.
Let $\theta'$ be the ordering of $\I D'$ induced by $\theta$; namely,
$\theta'(V):= \theta(V)-r_V$ where $r_V$ denotes the number of roots $R$
of $\I D$ such that $\theta(R)<\theta(V)$.  Let $F'_T$ be obtained from
$F_T$ by blowing up $\bigcup T_R$, and for each $R$, let $E'_R$ be the
preimage of $T_R$.  Set
  $$\mc I':= \mc I\mc O_{F'_T}\bigl(\textstyle\sum_R m_RE'_R\bigr).$$
Finally, let $n'$ be the number of vertices of $\I D'$.

Then $\mc I'$ generates geometrically complete ideals on $F'_\eta$ and
$F'_y$, each with diagram $\I D'$ owing to the theory of geometrically
complete ideals over a field; see Subsection~\ref{sb:ci}.  (To ensure
that the ideals on $F'_\eta$ and $F'_y$ have the same diagram, it is
necessary to omit all the roots of $\I D$.  Indeed, $\I D$ might have
two roots with the same multiplicity, but the diagram obtained by
omitting one root might differ from that obtained by eliminating the
other.  Conceivably, the two roots get interchanged under the
specialization.)

  Plainly, $\tilde{\I t}$ induces a $k(\eta)$-point $\tilde{\I t}'$ of
$F(\I D',\,\theta')$ such that $\mc I'_\eta$ generates an invertible
sheaf on the corresponding $F^{\prime(n'+1)}_\eta$, which is equal to
$F^{(n+1)}_\eta$.  Hence, by induction on $n$, we may assume that
$\tilde{\I t}'$ extends to $T$-point $\I t'$ of $F(\I D',\,\theta')$
such that, on the corresponding scheme $F^{\prime(n'+1)}_T$, the ideal
$\mc I'$ generates an invertible ideal.  It remains to show that $\I t'$
and the several isomorphisms $T_R\risom T$ yield an extension $\I t$ of
$\tilde{\I t}$.

Proceed by induction on $i$ where $0\le i\le n$.  Suppose we have
constructed a sequence $(t_0,\dotsc,t_{i-1})$ extending the sequence
$(\tilde t_0,\dotsc,\tilde t_{i-1})$ coming from $\tilde{\I t}$; suppose
also that, if we blow up $F^{(i)}_T$ along the preimage of
$\bigcup_{k\ge i} T_k$, then we get $F^{\prime(i')}_T$ where, for $0\le
j\le n$, we let $j'$ denote $j$ diminished by the number of roots $R$ of
$\I D$ such that $\theta(R)<j$.  Note that the base case $i:=0$ obtains:
the sequence $(t_0,\dotsc,t_{i-1})$ is empty; furthermore,
$F^{(i)}_T=F_T$ and $F^{\prime(i')}_T = F'_T$, which is the blowup of
$F_T$ along $\bigcup_{k\ge i} T_k$.

Note that $F^{(i)}_T\to F_T$ is an isomorphism off $\bigcup_{k< i} T_k$.
Indeed, given $j<i$, let $R'\in \I D$ be the root preceding
$\theta^{-1}(j)$, and set $k:=\theta(R')$.  Since $\theta$ is an
ordering, $k\le j$.  Since $(t_0,\dotsc,t_{i-1})$ extends $(\tilde
t_0,\dotsc,\tilde t_{i-1})$, the image of $T^{(j)}_\eta$ in $F_T$ is
just $(T_k)_\eta$.  So $T^{(j)}$ maps into $T_k$, and $k<i$.

Set $V:=\theta^{-1}(i)\in \I D$.  First suppose $V$ is a root of $\I D$.
Then $(i+1)'=i'$.  Also, $T_i$ is defined, and the isomorphism
$T_i\risom T$ provides a section $t_i$ of $F^{(i)}_T$ owing to the
preceding note.  By the same token, the blowup of $F^{(i+1)}_T$ along
the preimage of $\bigcup_{k\ge i+1} T_k$ is equal to the blowup of
$F^{(i)}_T$ along the preimage of $\bigcup_{k\ge i} T_k$.  But the
latter blowup is equal to $F^{\prime(i')}_T$.  It follows that $t_i$
does the trick.

Next suppose $V$ is not a root, so $V\in \I D'$.  Also $\bigcup_{k\ge i}
T_k = \bigcup_{k\ge i+1} T_k$.  Now, by the induction assumption,
$F^{\prime(i')}_T$ is equal to $F^{(i)}_T$ off the preimage of
$\bigcup_{k\ge i} T_k$.  Take $t_i:=t_i'$ where $(t_0',\dotsc,t_i')$
comes from $\I t'$. It is not hard to see that $t_i$ does the trick.

It is not immediately obvious that $(t_0,\dotsc,t_n)$ is strict, even
though $(t_0',\dotsc,t_{n'}')$ is strict.  However, $\I t$ is a
$T$-point of $F^{(n)}(T)$ and $\I t_\eta$ is a $k(\eta)$-point of $F(\I
D,\,\theta)$; furthermore, $\I t_y$ is a $k(y)$-point of $F(\I
D,\,\phi)$ for some ordering $\phi$ of $\I D$.  Since $T$ is
irreducible, $\I t_y$ is a point of the closure of $F(\I D,\,\theta)$ in
$F^{(n)}$, so is a point of $E(\I D,\,\theta)$.  Hence $E(\I
D,\,\theta)$ contains $E(\I D,\,\phi)$ by Proposition~\ref{prpCtd}.
But, by the initial reduction, $E(\I U,\,\theta)$ is minimal, so equal
to $E(\I D,\,\phi)$.  Hence $(\I D,\,\theta) \cong (\I D,\,\phi)$ again
by Proposition~\ref{prpCtd}.  So $\I t_y$ is a point of $F(\I
D,\,\theta)$.  Since $T$ is reduced, $\I t$ is therefore a $T$-point of
$F(\I D,\,\theta)$, as desired.
\end{proof}

\begin{dfn}\label{dfnCI}
Given an Enriques diagram $\I D$, say with $d:=\deg\I D$, let $H(\I
D)\subset \Hilb^d_{F/Y}$ denote the subset parameterizing the
geometrically complete ideals with diagram $\I D$ on the geometric
fibers of $F/Y$; see Subsection~\ref{sb:ci}.
\end{dfn}

\begin{prp}\label{prpF2H}
Let $\I D$ be an Enriques diagram, set $d:=\deg\I D$, and choose an
ordering $\theta$.  Then there exists a natural map $\Upsilon_\theta\:F(\I
D,\,\theta)\to \Hilb^d_{F/Y}$, whose formation commutes with base
extension of $Y$.  Its image is $H(\I D)$, and it factors into a finite
map $F(\I D,\,\theta)\to U$ and an open embedding $U\into
\Hilb^d_{F/Y}$.  Moreover, $\Upsilon_\theta= \Upsilon_{\theta'}\circ
\Phi_{\theta,\theta'}$ for any second ordering $\theta'$.
\end{prp}
\begin{proof}
Say $\I D$ has $n+1$ vertices $V$ with weights $m_V$.  On $F^{(n+1)}$,
set
	$$E:=\tsum_Vm_VE^{(\theta(V)+1,\,n+1)} \text{ and }
	 \mc L:=\mc O(-E).$$
Consider the standard short exact sequence:
        $$0\to\mc L\to\mc O_{F^{(n+1)}}\to\mc O_{E}\to0.$$
It remains exact on the fibers of $\pi^{(n+1)}\:F^{(n+1)}\to F^{(n)}$.
And $\pi^{(n+1)}$ is flat by Lemma~\ref{lem:DerFam}.  Hence $\mc L$ and
$\mc O_{E}$ are flat over $F^{(n)}$ owing to the local criterion.

Fix a $T$-point of $F(\I D,\,\theta)\subset F^{(n)}$.  It corresponds to
a strict sequence of arbitrarily near $T$-points of $F/Y$ by
Theorem~\ref{thm:3-2}.  Set $\vf:=\vf^{(1)}_T\dotsb\vf^{(n+1)}_T$.  Let
$t\in T$.  Then $R^i\vf_{t*}(\mc L_t)=0$ and $R^i\vf_{t*}(\mc
O_{F^{(n+1)}_t})=0$ for $i\ge1$ by \cite[Thm.~2.13, p.~22]{D73}.
Therefore, by Lemma \ref{lemFlat}, the induced sequence on $F_T$,
\begin{equation}\label{equationprpF2H1}
        0\to \vf_*\mc L_T\to \vf_*\mc O_{F^{(n+1)}_T}
                \to \vf_*\mc O_{E_T}\to0,
\end{equation}
 is an exact sequence of $T$-flat sheaves, and forming it commutes with
extending $T$.

The middle term in (\ref{equationprpF2H1}) is equal to $\mc O_{F_T}$:
the comorphism $\mc O_{F_T}\to \vf_*\mc O_{F^{(n+1)}_T}$ is an
isomorphism, since forming it commutes with passing to the fibers of
$F_T/T$, and on the fibers, it is an isomorphism as it is the comorphism
of a birational map between smooth varieties.  The third term in
(\ref{equationprpF2H1}) is a locally free $\mc O_T$-module of rank $d$
because its fibers are vector spaces of dimension $d$ owing again to
\cite[Thm.~2.13, p.~22]{D73}.  Therefore, (\ref{equationprpF2H1})
defines a $T$-point of $\Hilb^d_{F/Y}$.

The construction of this $T$-point is, plainly, functorial in $T$, and
commutes with base extension of $Y$.  Hence it yields a map
$\Upsilon_\theta \: F(\I D,\,\theta)\to \Hilb^d_{F/Y}$, whose formation
commutes with extension of $Y$.

To see that $H(\I D)$ is the image of $\Upsilon_\theta$, just
 observe that, in view of Subsection~\ref{sb:ci}, if $T$ is the spectrum
of an algebraically closed field, then $\vf_*\mc L_T$ is a geometrically
complete ideal on $F_T$ with diagram $\I D$, and every such ideal on
$F_T$ is of this form for some choice of $T$-point of $F(\I
D,\,\theta)$.

Let $\theta'$ be a second ordering.  Then by the construction of
$\Phi_{\theta,\theta'}$ in the proof of Proposition~\ref{prpIso}, our
$T$-point of $F(\I D,\,\theta)$ is carried to that of $F(\I
D,\,\theta')$ given by Lemma~\ref{lemAut} with
$\alpha:=\theta'\circ\theta^{-1}$.  Moreover, the lemma says that
$F_T^{(n+1)}$ is unchanged and implies that $E^{(\theta(V)+1,\,n+1)}=
E^{(\theta'(V)+1,\,n+1)}$ for all $V$.  Hence $\Upsilon_\theta=
\Upsilon_{\theta'}\circ \Phi_{\theta,\theta'}$.

 By Zariski's Main Theorem in the form of \cite[Thm.~(8.12.6),
p.~45]{EGAIV3}, there exists a factorization
      $$\Upsilon_\theta\: F(\I D,\,\theta)\xto{\Omega} H
	\xto{\Theta}\Hilb^d_{F/Y},$$
where $\Omega$ is an open embedding and $\Theta$ is a finite map.  Let
$W$ be the image of $\Omega$, so $\Theta(W)=H(\I D)$.  Replace $H$ by
the closure of $W$, and let us prove $W=\Theta^{-1}H(\I D)$.

Let $v\in\Theta^{-1}H(\I D)$.  Then $v$ is the specialization of a point
$w\in W$ since $H$ is the closure of $W$.  And $w$ is the image of a
point $\I w\in F(\I D,\,\theta)$.  Hence, by \cite[Thm.~(7.1.9),
p.~141]{EGAII}, there is a map $\tau\:T\to H$ where $T$ is the spectrum
of a discrete valuation ring, such that the closed point $y\in T$ maps
to $v$ and the generic point $\eta\in T$ maps to $w$; also there is a
$k(\eta)$-point $\tilde{\I t}$ of $F(\I D,\,\theta)$ supported at $\I
w$.

The map $\Theta\circ\tau$ corresponds to a coherent ideal $\mc I$ on
$F_T$.  Now, both $\Theta(w)$ and $\Theta(v)$ lie in $H(\I D)$; so $\mc
I$ generates geometrically complete ideals on $F_\eta$ and $F_y$, each
with diagram $\I D$.  And $\Upsilon_\theta(\tilde{\I t})$ corresponds to
$\mc I_\eta$ on $F_\eta$; so $\mc I_\eta$ generates an invertible sheaf
on $F^{(n+1)}_\eta$.  Hence, by Lemma~\ref{lemDVR1}, the $k(\eta)$-point
$\tilde{\I t}$ extends to $T$-point $\I t$ of $F(\I D,\,\theta)$.

Then $\Upsilon_\theta(\I t)\:T\to W$ carries $\eta$ to $w$.  But $H/Y$
is separated.  Hence $\Upsilon_\theta(\I t)=\tau$ by the valuative
criterion \cite[Prp.~(7.2.3), p.~142]{EGAII}.  But $\tau(y)=v$.  Hence
$v\in W$.  Thus $W\supset\Theta^{-1}H(\I D)$.  But $\Theta(W)=H(\I D)$.
Therefore, $W=\Theta^{-1}H(\I D)$.

But $W$ is open in $H$, and $\Theta$ is finite.  So $\Theta(H)$ and
$\Theta(H-W)$ are closed in $\Hilb^d_{F/Y}$.  Hence $H(\I D)$ is open in
$\Theta(H)$.  So there is an open subscheme $U$ of $\Hilb^d_{F/Y}$ such
that $U\bigcap \Theta(H)=H(\I D)$.  Furthermore, $W\to U$ is finite, as
it is the restriction of $\Theta$.  So $F(\I D,\,\theta)\to U$ is
finite.  The proof is now complete.
\end{proof}

\begin{cor}\label{corHLC}
Let $\I D$ be an Enriques diagram, and set $d:=\deg\I D$.  Then $H(\I
D)$ is a locally closed subset of $\Hilb^d_{F/Y}$.
\end{cor}
\begin{proof}
By Proposition~\ref{prpF2H}, $H(\I D)$ is the image of a finite map into
an open subscheme $U$ of $\Hilb^d_{F/Y}$.  So $H(\I D)$ is closed in
$U$, so locally closed in $\Hilb^d_{F/Y}$.
\end{proof}

\begin{rmk}\label{rmk}
 Lossen~\cite[Prp.~2.19, p.~35]{Los98} proved a complex analytic version
of Corollary~\ref{corHLC}.  Independently, Nobile and
Villamayor~\cite[Thm.~2.6, p.~250]{NV97} proved the corollary assuming
$\Hilb^d_{F/Y}$ is reduced and excellent; in fact, they worked with an
arbitrary flat family of ideals on a reduced excellent scheme, but of
course, any flat family is induced by a map to the Hilbert scheme.  All
three approaches are rather different.
\end{rmk}

\begin{thm}\label{thmF2H}
Let $\I D$ be an Enriques diagram, and set $d:=\deg\I D$.  Choose an
ordering $\theta$, and form the map $\Upsilon_\theta$ of
Proposition~{\rm\ref{prpF2H}}.  Then $\Upsilon_\theta$ induces a map
 $$\Psi\:F(\I D,\,\theta) \big/{\Aut}(\I D)\to \Hilb^d_{F/Y}.$$
It is universally injective; in fact, it is an embedding in
characteristic $0$.  Furthermore, $\Psi$ is independent of the
choice of $\theta$, up to a canonical isomorphism.
\end{thm}
\begin{proof}
By Corollary~\ref{corAut}, ${\Aut}(\I D)$ acts freely.  Hence, the
quotient map
 $$\Pi\:F(\I D,\,\theta)\to F(\I D,\,\theta)\big/{\Aut}(\I D)$$
 is faithfully flat.  By Proposition~\ref{prpF2H}, the action of
${\Aut}(\I D)$ is compatible with $\Upsilon_\theta$, and is compatible
with a second choice of ordering $\theta'$, up to the isomorphism
$\Phi_{\theta,\theta'}$.  Hence, by descent theory, $\Upsilon_\theta$
induces the desired map $\Psi$.  Plainly, its formation commutes
with base change.

Plainly, a map is universally injective if it is injective on geometric
points.  Furthermore, since $\Pi$ is surjective,
Proposition~\ref{prpF2H} also implies that $\Psi$ too factors
into a finite map followed by an open embedding.  Now, a finite map is a
closed embedding if its comorphism is surjective.  Hence, to prove that
$\Psi$ is an embedding, it suffices to prove that its fibers over
$Y$ are embeddings.  Now, forming $\Psi$ commutes with extending
$Y$.  Therefore, we may assume $Y$ is the spectrum of an algebraically
closed field $K$.

To prove $\Psi$ is universally injective, plainly we need only
prove $\Psi$ is injective on $K$-points.  Since $\Pi$ is
surjective, every $K$-point of $F(\I D,\,\theta)\big/{\Aut}(\I D)$ is
the image of a $K$-point of $F(\I D,\,\theta)$.  Hence we need only
observe that, if two $K$-points $\I t'$ and $\I t''$ of $F(\I
D,\,\theta)$ have the same image in $\Hilb^d_{F/Y}(K)$ under
$\Upsilon_\theta$, then the two differ by an automorphism $\gamma$ of
$\I D$.  But that image corresponds to a geometrically complete ideal
$\mc I$ on $F_K$ with diagram $\I D$.  In turn, as explained in
Subsection~\ref{sb:ci}, $\mc I$ determines a set $\I A$ of 2-dimensional
regular local $K$-domains whose fraction field is that of $F_K$, and $\I
A$ has a proximity structure, under which it is isomorphic to $\I D$.
Say ${\I t}'\in F(\I A,\,\theta')$ and ${\I t}''\in F(\I A,\,\theta'')$.
Then $\theta'^{-1}\circ\theta''$ induces the desired automorphism
$\gamma\in{\Aut}(\I D)$.

By Corollary~\ref{corQt}, $F(\I D,\,\theta) \big/{\Aut}(\I D)$ is smooth
and irreducible.  By Corollary~\ref{corHLC}, $H(\I D)$ is a locally
closed subset of $\Hilb^d_{F/Y}$, so carries an induced reduced
structure.  And $\Psi$ induces a bijective finite map $\beta\:F(\I
D,\,\theta) \big/{\Aut}(\I D) \to H(\I D)$.

Suppose $K$ is of characteristic $0$.  Then $\beta$ is birational.  If,
perchance, $\I D$ is minimal in the sense of \cite[Section~2,
p.~213]{KP99}, then $H(\I D)$ is smooth by the direct, alternative proof
of \cite[Prp.~(3.6), p.~225]{KP99}; hence, $\beta$ is an isomorphism.
In any case, it follows from Proposition 3.3.14 on p.~70 of \cite{Gu98}
that $\beta$ is unramified; hence, $\beta$ is an isomorphism.  The proof
is now complete.
\end{proof}

\begin{cor}\label{corHsbs}
Fix an Enriques diagram $\I D$, and set $d:=\deg\I D$.  Assume the
characteristic is $0$.  Then $H(\I D)\subset \Hilb^d_{F/Y}$ supports a
natural structure of $Y$-smooth subscheme with irreducible geometric
fibers of dimension $\dim(\I D)$.
\end{cor}
\begin{proof}
 By Theorem~\ref{thmF2H}, $\Upsilon_\theta$ induces an embedding of
$F(\I D,\,\theta) \big/{\Aut}(\I D)$ into $\Hilb^d_{F/Y}$.  By
Proposition~\ref{prpF2H}, the image is $H(\I D)$.  And by
Corollary~\ref{corQt}, the source is $Y$-smooth, and has irreducible
geometric fibers of dimension $\dim(\I D)$.
\end{proof}

\begin{prp}\label{prpM}
Given positive integers $r_1, \ldots ,r_k$, let $G(r_i)\subset
\Hilb^{r_i}_{F/Y}$ be the open subscheme over which the universal family
is smooth, and let
$$G(r_1,\ldots,r_k)\subset G(r_1)\times_Y \cdots \times_Y G(r_k)$$
be the open subscheme over which, for $i\neq j$, the fibers of the
universal families over $G(r_i)$ and $G(r_j)$ have empty intersection.
Set $r:=\sum r_i$.

Given distinct integers $m_1,\ldots ,m_k\ge 2$, let $\I D$ be the the
weighted Enriques diagram with $r$ vertices, each a root, and an
ordering $\theta$ such that the first $r_1$ vertices are roots of weight
$m_1$, the next $r_2$ are of weight $m_2$, and so on.  Set
$d:=\sum\binom{m_i+1}{2}r_i$.

Then $F(\I D,\,\theta)$ is equal to the complement in the relative
direct product $F^{\x_Y r}$ of the $\binom r2$ large diagonals, and $F(\I
D,\,\theta) \big/{\Aut}(\I D)$ is equal to 
 $G(r_1,\ldots,r_k)$.
Further, $\Upsilon_\theta$ always induces an embedding
\[\Psi\colon G(r_1,\ldots,r_k) \into \Hilb^{d}_{F/Y};\] 
 on $T$-points, $\Psi$ acts by taking a $k$-tuple
 $(W_1,\dotsc,W_k)$ where $W_i$ is a smooth
length-$r_i$ subscheme of $F_T$, say with ideal $\mc I_i$, to the
length-$d$ subscheme $W$ with ideal $\prod\mc I^{m_i}_i$.
\end{prp}

\begin{proof}
  Let $(t_0,\dotsc,t_{r-1})$ be a strict sequence of arbitrarily near
  $T$-points of $F/Y$ with diagram $(\I D,\,\theta)$.  Plainly, the
  $t_i$ are just sections of $F_T$, and their images are disjoint.  So
  $F(\I D,\,\theta)$ is equal to the asserted complement.

Plainly, ${\Aut}(\I D)$ is the product of $k$ groups, the $i$th being
the full symmetric group on the $r_i$ roots in the $i$th set.  So the
quotient $F(\I D,\,\theta) \big/{\Aut}(\I D)$ is equal to the open
subscheme of $\Hilb^{r_1}_{F/Y}\times_Y \cdots \times_Y
\Hilb^{r_k}_{F/Y}$ whose geometric points parameterize the $k$-tuples
whose $i$th component is an unordered set of $r_i$ geometric points of
$F$ such that all $r$ points are distinct; in other words, the quotient
is equal to the asserted open subscheme.

Since each vertex is a root of some weight $m_i$, plainly $\Psi$ acts
on $T$-points in the asserted way, owing to the following standard
general result, which is easily proved by descending induction: let $A$
be a locally Noetherian scheme, $\mc I$ a regular ideal, $b\:B\to A$ the
blow-up of $\mc I$, and $E$ the exceptional divisor; let $m\ge0$ and set
$\mc L:=\mc O_B(-mE)$; then $R^qb_*\mc L=0$ for $q\ge1$ and $b_*\mc
L=\mc I^m$.
 
Finally, to prove that $\Psi$ is always an embedding, we may
assume that $Y$ is the spectrum of an algebraically closed field $K$,
owing to the proof of Theorem~\ref{thmF2H}.  By the same token,
$\Psi$ is universally injective, and factors into a finite map
followed by an open embedding.  Hence, we need only show that
$\Psi$ is unramified.

Let $v$ be a $K$-point of $\Hilb^r_{F/Y}$; let $V\subset F$ be the
corresponding subscheme, and $\mc I$ its ideal.  Recall the definition
of the isomorphism from the tangent space at $v$ to the normal space
$\Hom(\mc I, \mc O_V)$; the definition runs as follows.  Let
$K[\epsilon]$ be the ring of dual numbers, and set
$T:=\Spec(K[\epsilon])$.  An element of the tangent space corresponds to
a $T$-point of $\Hilb^r_{F/Y}$ supported at $v$; so it represents a
$T$-flat subscheme $V_\epsilon\subset F_T$ that deforms $V$.  The
natural splitting $K[\epsilon] = K \oplus K\epsilon$ induces a splitting
$\mc O_{V_\epsilon} = \mc O_V \oplus \mc O_V\epsilon$.  Similarly, the
ideal $\mc I_\epsilon$ of $V_\epsilon$ splits: $\mc I_\epsilon = \mc I
\oplus \mc I\epsilon$.  Then the natural map $\mc O_{F_T}\to \mc
O_{V_\epsilon}$ restricts to a map $\mc I\to \mc O_V\epsilon$, which is
equal to the desired map  $\zeta\:\mc I\to \mc O_V$.

Assume $v\in G(r_1,\ldots,r_k)$.  Then $V$ is the union of $k$ sets of
reduced $K$-points of $F$.  The $i$th set has $r_i$ points; let $\mc
I_i$ be the ideal of its union.  Further, $\Psi$ carries $V$ and
$V_\epsilon$ to the subschemes $W$ and $W_\epsilon$ defined by $\mc
I_1^{m_1} \dotsb \mc I_k^{m_k}$ and $\mc I_{1,\epsilon}^{m_1} \dotsb \mc
I_{k,\epsilon}^{m_k}$.  So $\Psi$ is unramified at $v$ if the induced
map on tangent spaces is injective:
$$\psi\:T_{G(r_1,\ldots,r_k),v}\into \Hom(\mc I_1^{m_1}\cdots
 \mc I_k^{m_k},\, \mc O_W).$$

Say $v=(v_1,\ldots,v_k)$ with $v_i\in G(r_i)$, and say $v_i$ represents
$V_i\subset F$.  Then
\[T_{G(r_1,\ldots,r_k),v} =\bigoplus T_{\Hilb_{F/Y}^{r_i},v_i}
 =\bigoplus\Hom(\mc I_i,\mc O_{V_i}).\]
Given any $\zeta\in T_{G(r_1,\ldots,r_k),v}$, its image $\psi(\zeta)$ is
equal to the restriction of the canonical map $\mc O_{F_T}\to \mc
O_{W_\epsilon}$.  So $\psi$ splits into a direct sum of local components
  $$\psi_x\:\Hom(\mc I_{i,x}, \mc O_{V,x})\to \Hom(\mc I^{m_i}_{i,x},
\mc O_{W,x})
\quad\text{for }x\in V_i\text{ and } i=1,\dotsc,k.$$
It remains to prove that each $\psi_x$ is injective.  Fix an $x$.

Set $\mc I:=\mc I_i$ and $m:=m_i$.  Fix generators
$\mu,\nu\in\mc I_x$.  Set $a:=\zeta_x\mu$ and $b:=\zeta_x\nu$ in $\mc O_{V_i,x}=K$.
Then $\mc I_{\epsilon,x}$ is generated by $\mu-a\epsilon$ and
$\nu-b\epsilon$;  so $\mc I^m_{\epsilon,x}$ is generated by
\begin{multline*}
\mu^m-m\mu^{m-1}a\epsilon,\ \mu^{m-1}\nu-(m-1)\mu^{m-2}\nu
a\epsilon-\mu^{m-1}b\epsilon,\ \dotsc,\\
\mu\nu^{m-1}-a\nu^{m-1}\epsilon-(m-1)b\mu\nu^{m-2}\epsilon,\ \nu^m-m\nu^{m-1}b\epsilon.
\end{multline*}
Hence, modulo $\mc I^m_{\epsilon,x}$, the generators $\mu^{m-1}\nu$ and
$\mu\nu^{m-1}$ of $\mc I^m_x$ are congruent to $(m-1)\mu^{m-2}\nu
a\epsilon+\mu^{m-1}b\epsilon$ and $a\nu^{m-1}\epsilon
+(m-1)b\mu\nu^{m-2}\epsilon$.  (They're equal if $m=2$.) 

Form the latter's classes in $\mc
O_{W,x}$.  Then, therefore, these classes are the images of those
generators under the map $\psi_x\zeta_x$.  Hence, in any
characteristic, we can recover $a$ and $b$ from the images of
$\mu^{m-1}\nu$ and $\mu\nu^{m-1}$.  But $a$ and $b$ determine $\zeta_x$.
Thus $\psi_x$ is injective, and the proof is complete.
\end{proof}

\appendix
\section{Generalized property of exchange}\label{ap:exchg}
This appendix proves two lemmas of general interest, which we need.  The
first lemma generalizes the property of exchange to a triple $(T,f,F)$
where $T$ is a (locally Noetherian) scheme, $f\:P\to Q$ is a proper map
of $T$-schemes of finite type, and $\mc F$ is a $T$-flat coherent sheaf
on $P$.  The original treatment was made by Grothendieck and Dieudonn\'e
in \cite[Sec.~7.7, pp.~65--72]{EGAIII2}, and somewhat surprisingly,
deals only with the case of $Q=T$.  (Although they replace $F$ by a
complex of flat and coherent sheaves bounded below, this extension is
minor and we do not need it.)

The first lemma is proved by generalizing the treatment in Section II, 5
of \cite[pp.\,46--55]{Mu70}.  Alternatively, as Illusie pointed out in a
private conversation, the lemma can be proved using the methods that he
developed in  \cite{Il70}.

The first lemma is used to prove the second.  The second is used
in the proof of Proposition~\ref{prpF2H}, which constructs the map from
the scheme of $T$-points with given Enriques diagram to the Hilbert
scheme.

\begin{lem}[Generalized property of exchange]\label{lemExch}
Let $T$ be a scheme, $f\:P\to Q$ a proper map of  $T$-schemes of finite
type, and $\mc F$ a $T$-flat coherent sheaf on $P$.  Let $q\in Q$ be a
point, $t\in T$ its image, and  $i\ge0$ an integer.  Assume that, on the
fiber $Q_t$, the base-change map  of sheaves 
\begin{equation*}\label{eqExch1}
        \rho^i_t\:(R^if_*\mc F)_t \to R^if_{t*}\mc F_t
\end{equation*}
is surjective at $q$.  Then there exists a neighborhood $U$ of $q$ in
$Q$ such that, for any $T$-scheme $T'$, the base-change map of sheaves
\begin{equation*}\label{eqExch2}
        \rho^i_{T'}\:(R^if_*\mc F)_{T'} \to R^if_{{T'}*}\mc F_{T'}
\end{equation*}
is bijective on the open subset $U_{T'}$ of $Q_{T'}$.  Furthermore, the
map $\rho^{i-1}_t$ is also surjective at $q$ if and only if sheaf
$R^if_*\mc F$ is  $T$-flat at $q$. 
\end{lem}
\begin{proof}
 The question is local on $Q$; so we may assume that $T=\Spec A$ and
$Q=\Spec B$ where $A$ is a Noetherian ring and $B$ is a finitely
generated $A$-algebra.  Also, we may assume that $B$ is $A$-flat by
expressing $B$ as a quotient of a polynomial ring over $A$ and then
replacing $B$ with that ring.  For convenience, when given a $B$-module
or a map of $B$-modules, let us say that it has a certain property at
$q$ to mean that it acquires this property on localizing at the prime
corresponding to $q$.

There is a finite complex $K^\bu$ of $A$-flat finitely generated
$B$-modules, and on the category of $A$-algebras $C$, there is, for
every $j\ge0$, an isomorphism of functors
\begin{equation*}\label{eqExch3}
 \uH^j(K^\bu\ox_AC) \risom \uH^j(P\ox_AC,\ \mc F\ox_AC).
\end{equation*}
Indeed, this statement results, mutatis mutandis, from
the proof of
the theorem on page 46 of \cite{Mu70}.

Let $k$ be the residue field of $t$.  Then there is a natural map of
exact sequences
\begin{equation}\label{CDExch1}\begin{CD}
K^{i-1}\ox k @>>> \uZ^i(K^\bu)\ox k @>>> \uH^i(K^\bu)\ox k @>>> 0\\
        @VV1V               @VV z^i_k V               @VV h^i_k V\\
K^{i-1}\ox k @>>> \uZ^i(K^\bu\ox k) @>>> \uH^i(K^\bu\ox k) @>>> 0\rlap{.}
\end{CD} \end{equation}
Since $\rho^i_k$ is surjective at $q$, so is $h^i_k$.  Hence $z^i_k$ is
surjective at $q$.

Consider the following map of exact sequences:
\begin{equation*}\label{CDExch2}\begin{CD}
\uZ^i(K^\bu)\ox k @>>> K^i\ox k @>>> \uB^{i+1}(K^\bu)\ox k @>>> 0\\
         @VV z^i_k V       @VV1V                @VV b^{i+1}_k V\\
\uZ^i(K^\bu\ox k) @>>> K^i\ox k @>>> \uB^{i+1}(K^\bu\ox k) @>>> 0\rlap{.}
\end{CD} \end{equation*}
Now, $z^i_k$ is surjective at $q$.  Hence $b^{i+1}_k$ is bijective at $q$.

Hence $\uB^{i+1}(K^\bu)\ox k\to K^{i+1}\ox k$ is injective at $q$.  Set
$L:=K^{i+1}/\uB^{i+1}(K^\bu)$.  Since $K^{i+1}$ is $A$-flat,
the local criterion of flatness implies that $L$ is $A$-flat at
$q$.  Hence, by the openness of flatness, there is a $g\in B$ outside
the prime corresponding to $q$ such that the localization $L_g$ is
$A$-flat.  We can replace $B$ by $B_g$, and so assume $L$ is $A$-flat.

Let $C$ be any $A$-algebra.  Then the following sequence is exact:
\begin{equation}\label{eqExch4}
0\to\uZ^i(K^\bu)\ox C \to K^i\ox C
 \to K^{i+1}\ox C \to L\ox C  \to0.
\end{equation}
It follows that, in the map of exact sequences
\begin{equation*}\label{CDExch3}\begin{CD}
K^{i-1}\ox C @>>> \uZ^i(K^\bu)\ox C @>>> \uH^i(K^\bu)\ox C @>>> 0\\
        @VV1V               @VV z^i_C V               @VV h^i_C V\\
K^{i-1}\ox C @>>> \uZ^i(K^\bu\ox C) @>>> \uH^i(K^\bu\ox C) @>>> 0\rlap{,}
\end{CD} \end{equation*}
$z^i_C$ is bijective.  Hence $h^i_C$ is bijective.  Thus the first
assertion holds: $\rho^i_C$ is bijective.

If $\uH^i(K^\bu)$ is  $A$-flat at $q$, then plainly the sequence
\begin{equation}\label{eqExch5}
0\to \uB^i(K^\bu)\ox k\to \uZ^i(K^\bu)\ox k\to \uH^i(K^\bu)\ox k\to0
\end{equation}
is exact.  The converse holds too by the local criterion for
flatness, because $\uZ^i(K^\bu)$ is $A$-flat  owing to the exactness of
(\ref{eqExch4}) with $C:=A$ and to the flatness of $L$. 

Since $z^i_k$ is bijective, (\ref{eqExch5}) is exact if and only if
$b^i_k$ is injective.  The latter holds if and only if $z^{i-1}_k$ is
surjective, owing to the map of  exact sequences
\begin{equation*}\label{CDExch4}\begin{CD}
 @. Z^{i-1}(K^\bu)\ox k  @>>> K^{i-1}\ox k @>>> \uB^i(K^\bu)\ox k  @>>> 0\\
  @.    @VV z^{i-1}_kV           @VV 1 V              @VV b^i_k V\\
 0 @>>> Z^{i-1}(K^\bu\ox k)  @>>> K^{i-1}\ox k @>>> \uB^i(K^\bu\ox k)
   @>>> 0\rlap{.}
\end{CD} \end{equation*}

Finally, $z^{i-1}_k$ is surjective if and only if $h^{i-1}_k$ is so,
owing to (\ref{CDExch1}) with $i-1$ in place of $i$.  Putting it all
together, we've proved that $h^{i-1}_k$ is surjective if and only if
$\uH^i(K^\bu)$ is $A$-flat at $q$.  In other words, the second assertion
holds too.
\end{proof}

\begin{lem}\label{lemFlat}
 Let $T$ be a scheme, $f\:P\to Q$ a proper map of $T$-schemes of finite
type, and
\begin{equation}\label{eqFlat1}
        0\to\mc F\to\mc G\to\mc H\to0
\end{equation}
 a short exact sequence of $T$-flat coherent sheaves on $P$.  For each
point $t\in T$, let $f_t$ and $\mc F_t$ and $\mc G_t$ denote the
restrictions to the fiber $P_t$, and assume that
\begin{equation}\label{eqFlat2}
  R^if_{t*}(\mc F_t)=0\text{ and } R^if_{t*}(\mc G_t)=0\text{ for }
i\ge1.
\end{equation}
 Then the induced sequence on $Q$,
\begin{equation}\label{eqFlat3}
 0\to f_*\mc F\to f_*\mc G\to f_*\mc H\to0,
\end{equation}
 is a short exact sequence of $T$-flat coherent sheaves, and forming it
commutes with base extension.
\end{lem}
\begin{proof}
 Since $\mc H$ is $T$-flat, the sequence (\ref{eqFlat1}) remains exact
after restriction to the fiber $P_t$ for each $t\in T$, and so the
restricted sequence induces a long exact sequence of cohomology.  Hence,
(\ref{eqFlat2}) yields
\begin{equation*} \label{eqFlat6}
 R^if_{t*}(\mc H_t)=0\hbox{ for all }i\ge1.
\end{equation*}
 
By hypothesis, $\mc F$, $\mc G$, $\mc H$ are $T$-flat.  Hence, by the
generalized property of exchange, Lemma~\ref{lemExch}, the sheaves
$f_*\mc F$, $f_*\mc G$, $f_*\mc H$ are $T$-flat, and forming them
commutes with extending $T$.  By the same token, $R^1f_{*}(\mc F)=0$;
whence, Sequence  (\ref{eqFlat3}) is exact.  The assertion follows.
\end{proof}

\section{A few examples by Ilya TYOMKIN}
\label{ap:eg}

Let $F$ be the affine plane over the spectrum $Y:=\Spec(K)$ of
an algebraically closed field $K$ of positive characteristic $p$.  In
this appendix, we analyze a few simple examples of minimal Enriques
diagrams $\I D$.  Some depend
on $p$, and have an ordering $\theta$ for which the universally injective
map of Theorem~\ref{thmF2H},
 $$\Psi\:F(\I D,\,\theta) \big/{\Aut}(\I D)\to \Hilb^d_{F/Y},$$
is purely inseparable.  Others are independent of $p$; they have several
vertices, but only one root, yet they have an ordering $\theta$ for which
$\Psi$ is an embedding.  In fact, in every case, $\theta$ is unique, and
$\Aut(\I D)$ is trivial.

We take $F$ to be the affine plane just to simplify the presentation.
With little modification, everything works for any smooth irreducible surface
$F$.

It is unknown what conditions on an arbitrary Enriques diagram $\I D$
serve to guarantee here that $\Psi$ is unramified, so an embedding.
Nevertheless, in view of the analysis in this appendix, it is reasonable
to make the following guess.

\begin{gss}\label{gss2}
 If $p>\frac12\sum_{V\in\I D}\,m_V$, then $\Psi$ is  unramified.
\end{gss}

This guess is sharp in the sense that, if $p\le\frac12\sum_{V\in\I
  D}\,m_V$, then $\Psi$ may be ramified.  For example, consider the
plane curve $C:x_2^p=x_1^{p+1}$.  In the notation of Definition
\ref{dfMpm}, the minimal diagram of $C$ is $\I M_{p,p}$.  It has $p+1$
vertices with $m_V=p,1,1,\dotsc,1$.  So $p=\frac12\sum_{V\in\I D}
\,m_V$.  And $\Psi$ is ramified by Proposition \ref{prop:purins}.

Similarly, consider $C:y(y-x^p)=0$.  Its minimal diagram has $p$
vertices $V$ with $m_V=2$.   So $p=\frac12\sum_{V\in\I D}
\,m_V$.  And $\Psi$ is ramified by an argument similar to the proof of
Proposition \ref{prop:purins}.

On the other hand, if $\I D$ has a single vertex of weight $2p$, then
$\Psi$ is unramified by Proposition \ref{prpM}, and of course, $p=\frac12
\sum_{V\in\I D}\,m_V$.

In general, if a branch has tangency of order divisible by $p$ to an
exceptional divisor $E$, then the multiplicity of the root must be at
least $p$ and there must be at least $p$ other vertices. So $p\le
\frac12\sum_{V\in\I D}\,m_V$.  Instead, if, at a point $P\in F$, all the
branches have a tangency of order divisible by $p$ to the same smooth
curve $D$, then there must be at least $p$ vertices $V$ with $m_V\ge2$.
So again, $p\le\frac12\sum_{V\in\I D}\,m_V$.  Thus, if we guess that
$\Psi$ can be ramified in only these two ways, then we arrive at Guess~\ref{gss2}.

Further, although $\Psi$ does not sense first-order deformations either
along $E$ or along $D$, nevertheless after we add a transverse branch at
$P$, then $\Psi$ does sense first-order deformations of the new branch;
thus $\Psi$ becomes unramified.  This intuition is developed into a
rigorous proof for the ordinary tacnode in Proposition
\ref{prop:IsoForAll}, and a similar procedure works if the tacnode is
replaced by an ordinary cusp.

\begin{dfn}\label{dfMpm}
Fix $m\ge p$.  Let $\I M_{p,m}$ denote the minimal Enriques diagram of
the plane curve singularity with $1+m-p$ branches whose tangent lines
are distinct, whose first branch is $\{\,x_2^p=x_1^{p+1}\,\}$, and whose
remaining $m-p$ branches are smooth.
\end{dfn}
\begin{figure}
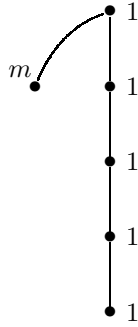

\[
\xy 
(0,0)*{\bullet}="A"; (10,10)*{\bullet}="B"; (10,0)*{\bullet}="C"; (10,-10)*{\bullet}="D"; (10,-20)*{\bullet}="E"; (10,-30)*{\bullet}="F";
"A"; "B" **\crv{(3,8)}; 
"B"; "C" **\dir{-}; 
"C"; "D" **\dir{-}; 
"D"; "E" **\dir{-}; 
"E"; "F" **\dir{-}; 
"A"+(-2,2)*{m};
"B"+(3,0)*{1};
"C"+(3,0)*{1};
"D"+(3,0)*{1};
"E"+(3,0)*{1};
"F"+(3,0)*{1};
\endxy 
\]
\caption{The Enriques diagram $\textbf M_{m,p}$, with $m\ge p=5$, of Definition B.2.}
\end{figure}

\begin{eg}\label{egch2}
  For motivation, consider the following special case.  Take $p:=2$ and
  $m:=2$.  Then $\I M_{p,m}$ is the minimal Enriques diagram $\I A_2$ of
  the cuspidal curve $C:x_2^2=x_1^3$.  This diagram has three vertices
  and a unique ordering $\theta$.

  Take $F:=\AAA^2_K$ and $T:=\Spec(K)$.  In $F(\mathbf
A_2,\theta)\subset F^{(2)}$, form
  the locus $L$ of sequences $(t_0,t_1,t_2)$ of arbitrarily near
  $T$-points of $F/K$ such that $t_0$ is the constant map from $T$ to
  the origin.  Plainly, the second projection induces an isomorphism
  $L\risom E'_K$ where $E'_K$ is the exceptional divisor of the blow up
  $F'_K$ of $F$ at the origin.

The strict transform $C'$ of $C$ is tangent to $E'_K$ with order 2, and
$C'$ is given by the equation $s^2=x_1$ where $s:=x_2/x_1$.  Notice that
this equation is preserved by any first order deformation along $E'_K$
of the point of contact; indeed,
 $$(s+b\epsilon)^2=s^2$$
as  $p=2$ and $\epsilon^2=0$. This observation
suggests that the restriction of $\Psi$,
 $$(\Psi|L)\: L\to {\Hilb}^5_{F/K},$$
 is purely inseparable; and indeed,  $\Psi|L$  is so, as we check next.

 Let $\I D'$ be the diagram obtained from $\I A_2$ by omitting the root,
 let $\theta'$ be the unique ordering of $\I D'$, and consider the
 corresponding map
 $$\Psi'\:F'(\I D',\,\theta') \to \Hilb^2_{F'_K/K}.$$
Plainly, the projection $(t_0,t_1,t_2)\mapsto (t_1,t_2)$ embeds $L$ into
$F'(\I D',\,\theta')$.

So $\Psi'$ induces a map $\Psi'_L\:L\to \Hilb^2_{F'_K/K}$.  It carries
$(t_0,t_1,t_2)$ to the subscheme  of $F'_T$ with ideal $\mc I'$ defined
by the formula
$$\mc I':=\bigl(\vf^{(2)}_T\vf^{(3)}_T\bigr)_*\mc O_{F^{(3)}_T}
 \bigl(-E^{(2,3)}_T-E^{(3,3)}_T\bigr).$$
But $E^{(2,3)}_T+E^{(3,3)}_T\le E^{(1,3)}_T$.  So
$$\mc O_{F^{(3)}_T} \bigl(-E^{(2,3)}_T-E^{(3,3)}_T\bigr)
 \supseteq \mc O_{F^{(3)}_T} \bigl(-E^{(1,3)}_T\bigr).$$
 Hence $\mc I'$ contains the ideal of $E'_T$.  Therefore, $\Psi'_L$ factors
 through $\Hilb^2_{E'_K/K}$, which is isomorphic to $\Sym^2(L)$.  The
 corresponding map $L\to\Sym^2(L)$  is the diagonal map
since $\Psi'_L(t_0,t_1,t_2)$ has the same
 support as $t_1$.  This diagonal map is purely inseparable as $p=2$.

Finally, $\Psi'_L\:L\to \Hilb^2_{E'_K/K}$ is a factor of $\Psi|L$
because $\Psi(t_0,t_1,t_2)$ is the subscheme of $F_T$ with ideal
$\bigl(\vf^{(1)}_T\bigr)_*\mc I'(2E'_T)$.  Thus $\Psi|L$ is, indeed,
purely inseparable.  In fact, $\Psi$ is purely inseparable by
Proposition~\ref{prop:purins} below.
\end{eg}

\begin{prp}\label{prop:purins}
Fix $m\ge p$.  Set $\I D:=\I M_{p,m}$ and $d:=\binom{m+1}{2}+p$.  Then $\I
D$ has a unique ordering $\theta$; also ${\Aut}(\I D)=1$ and $\deg\I
D=d$.  Take $F=\AAA^2_K$.  Then $\dim F(\I D,\,\theta)=3$, and
$\Upsilon_\theta\:F(\I D,\,\theta) \to \Hilb^d_{F/Y}$ is purely
inseparable; also, $\Psi =\Upsilon_\theta$.
\end{prp}

\begin{proof}
  Plainly, $\I D$ has $p+1$ vertices, say $V_0,\dotsc, V_p$ ordered by
  succession.  Then proximity is given by $V_k\succ V_{k-1}$ and
  $V_k\succ V_0$ for  $k>0$.  Further, the weights are given by
  $m_{V_0}=m$ and $m_{V_k}=1$ for $k>0$.  Set $\theta(V_k):=k$;
  plainly, $\theta$ is an ordering of $\I D$, and is the only one.
  Also, plainly, ${\Aut}(\I D)=1$ and $\deg\I D=d$.

Theorem~\ref{thm:3-2} says that $\dim F(\I D,\,\theta)=\dim\I D$, but
plainly $\dim\I D=3$.  Now, $\Psi = \Upsilon_\theta$ because ${\Aut}(\I
D)=1$.  Further, Theorem~\ref{thmF2H} says that $\Psi$ is universally
injective.  Hence $\Psi$ is purely inseparable, because it is everywhere
ramified owing to the following lemma.
\end{proof}

\begin{lem}\label{lem:ker}
Under the conditions of Proposition~{\rm\ref{prop:purins}}, let $t\in
F(\I D,\,\theta)$ be a $K$-point.  Then $\Ker(d_t\Upsilon_\theta)$ is of
dimension $1$.
\end{lem}

\begin{proof}
Say $t$ represents the sequence $(t_0,\dotsc, t_p)$ of arbitrarily near
$K$-points of $F/Y$.  Choose coordinates $x_1,\ x_2$ on $F$ such that
 $t_0:x_1=x_2=0$ and such that $t_1$ is the point of intersection of
the exceptional divisor $E_0$ with the proper transform of the
$x_1$-axis.  Set $s_0:=x_2/x_1$, set $s_1:=x_1/s_0$, and set
$s_k:=s_{k-1}/s_0$ for $2\le k\le p-1$.  Then $t_1:s_0=x_1=0$, and
$t_k:s_0=s_{k-1}=0$ for $2\le k\le p$.

Set $z:=\Upsilon_\theta(t)\in \Hilb^d_{F/Y}(K)$.  Let $Z$ denote the
corresponding subscheme, and $\CI$ its ideal.  Recall from the proof of
Proposition~\ref{prpF2H} that $\CI=\varphi_{K*}\CO(-E_K)$ where
$E_K=\sum_{i=0}^p m_{V_i}E^{(i+1,\,p+1)}$.  Recall
from the proof of Proposition~\ref{prop:purins} that $m_{V_0}=m$ and
$m_{V_k}=1$ for $k>0$ and that $V_k\succ V_0$ for  $k>0$.  It follows
that
$$E_K=me^{(1,\,p+1)}_KE^{(1)}+\sum_{k=1}^{p}k(m+1)e^{(k+1,\,p+1)}_KE^{(k+1)}_K.$$

Set $\delta(r):=0$ if $0\le r< p$ and $\delta(r):=1$ if $p\le r\le m$.
Set
 $$f_r:=x_1^{m+1-r-\delta(r)}x_2^{r}\text{\quad for } 0\le r\le m.$$
Let's now show that the $f_r$ generate $\CI$.

First, note that, for each  $r$ and for  $1\le k\le p-1$,
 $$f_r
=x_1^{m+1-\delta(r)}s_0^r=s_k^{m+1-\delta(r)}s_0^{k(m+1-\delta(r))+r}.$$
  Hence, the pullback of
$f_r$ vanishes along $e^{(1,\,p+1)}_KE^{(1)}_K$ to order at least $m$, and
along $e^{(k+1,\,p+1)}_KE^{(k+1)}_K$ to order at least $k(m+1)$ for
$k\ge 1$, since $r-k\delta(r)\ge 0$.  Thus
$f_r\in \CI$ for each $r$.

Let $\CJ$ be the ideal generated by the $f_r$.  Then $\CJ\subset\CI$.
Now, $K[x_1,x_2]/\CJ$ is spanned as a $K$-vector space by the monomials
$x_1^{m+1-r-\delta(r)}x_2^{l}$ for $0\le l<r\le m$ and by
$x_1^{m+1-p}x_2^l$ for $0\le l<p$.  Hence $\CJ=\CI$ because
 $$\dim \bigl(K[x_1,x_2]/\CJ\bigr)\le\tsum_{r=0}^{m}r+p= d =\dim
K[x_1,x_2]/\CI.$$

Let $K[\epsilon]$ be the ring of dual numbers, and set
$T:=\Spec(K[\epsilon])$.  Let $(t'_0,\dotsc,t'_p)$ be a strict sequence of
arbitrarily near $T$-points of $F/Y$ lifting $(t_0,\dotsc,t_p)$. Then there
are $a_1, a_2, b\in K$ so that, after setting
 $x_1':=x_1+a_1\epsilon$ and $x_2':=x_2+a_2\epsilon$ and
setting $s_0':=x_2'/x_1'+b\epsilon$ and $s_1':=x_1'/s_0'$ and
$s_k':=s_{k-1}'/s_0'$ for $2\le k\le p-1$,
 we have  $t_0':x_1'=x_2'=0$ and $t_1':s_0'=x_1'=0$ and
$t_k':s_0'=s_{k-1}'=0$ for $2\le k\le p$.

Let $t'\in F(\I D,\,\theta)(T)$ represent $(t'_0,\dotsc,t'_p)$.  Set
$z':=\Upsilon_\theta(t')\in \Hilb^d_{F/Y}(K)(T)$.  Let $Z'$ denote the
corresponding subscheme, and $\CI'$ its ideal.  Let's show that $\CI'$
is generated by the following elements:
$$f_r':=(x_1')^{m+1-r-\delta(r)}(x_2')^{r}\text{\quad for } 0\le r\le m.$$
The $f_r'$ reduce to the $f_r$, which generate $\CI$.  Further, $\CI'$
reduces to $\CI$ as $Z'$ is flat over $K[\epsilon]$.  Hence it
suffices to prove that $\CI'$ contains the $f_r'$.

Note that $(s'_0-b\epsilon)^p=(s'_0)^p$ as the characteristic is $p$.
Hence, for each  $r$,
$$f'_r=(x'_1)^{m+1-\delta(r)}(s'_0-b\epsilon)^r
      =(s'_k)^{m+1-\delta(r)}(s'_0)^{k(m+1)
        +(p-k)\delta(r)}(s'_0-b\epsilon)^{r-p\delta(r)}$$
for $1\le k\le p-1$.  Therefore, the pullback of $f'_r$ vanishes
along $e^{(1,\,p+1)}_TE^{(1)}_T$ to order at least $m$, and along
$e^{(k+1,\,p+1)}_TE^{(k+1)}_T$ to order at least $k(m+1)$ for $k\ge
1$ since $(p-k)\delta(r)\ge 0$ and $r-p\delta(r)\ge 0$.  Thus
$\CI'$ contains the $f_r'$.

Recall that $T_z{\rm  Hilb}^d_{F/Y}(K)=\Hom(\CI,\CO_Z)$. Furthermore, it follows from
the computations above that
$$d_t\Upsilon_\theta(t')(f_r') =  (m+1-r-\delta(r))x_1^{m-r-\delta(r)}x_2^{r}a_1
 +rx_1^{m+1-r-\delta(r)}x_2^{r-1}a_2$$
for  $0\le r\le m$.   Therefore,
$$\ker(d_t\Upsilon_\theta)
=\bigl\{(a_1,a_2,b)\bigm|a_1=a_2=0\bigr\},$$
and we are done.
\end{proof}

\begin{dfn}
Fix $m\ge 3$. Let $\I N_m$ denote the minimal Enriques diagram of the
following plane curve singularity: an ordinary tacnode
$\{\,x_2(x_2-x_1^2)=0\,\}$ union with $m-2$ smooth branches whose
tangent lines are distinct and different from the common tangent line of
the two branches of the tacnode.
\end{dfn}

\begin{prp}\label{prop:IsoForAll}
Fix $m\ge 3$. Set $\I D:=\I N_m$ and $d:=\binom{m+1}{2}+3$. Then $\I D$
has a unique ordering $\theta$; also $\Aut(\I D) = 1$ and $\deg \I D =
d$. Take $F=\AAA^2_K$.  Then $\dim F(\I D, \theta) = 3$, and
$\Upsilon_\theta\colon F(\I D, \theta)\to \Hilb^d_{F/K}$ is an
embedding; also, $\Psi =\Upsilon_\theta$.
\end{prp}
\begin{proof}
Plainly, $\I D$ has $2$ vertices, say $V_0$ and $V_1$ ordered by succession.
Then proximity is given by $V_1\succ V_0$. Further, the weights
are given by $m_{V_0}=m$ and $m_{V_1}=2$. Set $\theta(V_k):=k$; plainly, $\theta$ is an
ordering of $\I D$, and is the only one. Also, plainly, $\Aut(\I D) = 1$ and $\deg\I D = d$.
Theorem~\ref{thm:3-2} says that $\dim F(\I D, \theta) = \dim \I D$, but plainly $\dim\I D = 3$. Now, $\Psi=\Upsilon_\theta$ because $\Aut(\I D) = 1$.
Further, Theorem~\ref{thmF2H} says that $\Psi$ is universally
injective. Hence $\Psi$ is an embedding because it is nowhere ramified owing
to the following lemma.
\end{proof}

\begin{lem}
Under the conditions of Proposition~{\rm \ref{prop:IsoForAll}}, let $t\in
F(\I D,\,\theta)$ be a $K$-point.  Then $\Ker(d_t\Upsilon_\theta)=0$.
\end{lem}

\begin{proof}
Say $t$ represents the sequence $(t_0, t_1)$ of arbitrarily near
$K$-points of $F/Y$.  Choose coordinates $x_1,\ x_2$ on $F$ such that
$t_0:x_1=x_2=0$ and such that $t_1$ is the point of intersection of
the exceptional divisor $E_0$ with the proper transform of the
$x_1$-axis.  Set $s:=x_2/x_1$. Then $t_1:s=x_1=0$.

Set $z:=\Upsilon_\theta(t)\in \Hilb^d_{F/Y}(K)$.  Let $Z$ denote the
corresponding subscheme, and $\CI$ its ideal.  Recall from the proof of
Proposition~\ref{prpF2H} that $\CI=\varphi_{K*}\CO(-E_K)$ where
$E_K=\sum_{i=0}^1 m_{V_i}E^{(i+1,\,2)}$.  Recall
from the proof of Proposition~\ref{prop:IsoForAll} that $m_{V_0}=m$ and
$m_{V_1}=2$ and that $V_1\succ V_0$.  It follows
that
$$E_K=me^{(1,\,2)}_KE^{(1)}+(m+2)E^{(2)}_K.$$

Set $\delta(0):=2$, set $\delta(1):=1$, and set $\delta(r):=0$ if $r\ge 2$. Set
 $$f_r:=x_1^{m-r+\delta(r)}x_2^{r}\text{\quad for } 0\le r\le m.$$
Let's now show that the $f_r$ generate $\CI$.

First, note that, for each  $r$,
 $$f_r=x_1^{m+\delta(r)}s^r.$$
  Hence, the pullback of
$f_r$ vanishes along $e^{(1,\,2)}_KE^{(1)}_K$ to order at least $m$, and
along $E^{(2)}_K$ to order at least $m+2$, since $m+r+\delta(r)\ge m+2$.  Thus
$f_r\in \CI$ for each $r$.

Let $\CJ$ be the ideal generated by the $f_r$.  Then $\CJ\subset\CI$.
Now, $K[x_1,x_2]/\CJ$ is spanned as a $K$-vector space by the monomials
$x_1^{m-r+\delta(r)}x_2^{l}$ for $0\le l<r\le m$ and by $x_1^{m-1}$,
$x_1^{m-1}x_2$, and $x_1^{m+1}$.  Hence $\CJ=\CI$ because
 $$\dim \bigl(K[x_1,x_2]/\CJ\bigr)\le\tsum_{r=0}^{m}r+3= d =\dim
K[x_1,x_2]/\CI.$$
Furthermore, the monomials $x_1^{m-1}$ and $x_1^{m-1}x_2$ and $x_1^{m+1}$, and
$x_1^{m-r+\delta(r)}x_2^{l}$ for $0\le l<r\le m$ form a basis of the
$K$-vector space $K[x_1,x_2]/\CI$.

Let $K[\epsilon]$ be the ring of dual numbers, and set
$T:=\Spec(K[\epsilon])$.  Let $(t'_0, t'_1)$ be a strict sequence of
arbitrarily near $T$-points of $F/Y$ lifting $(t_0, t_1)$. Then there
are $a_1, a_2,b\in K$ so that, after setting
 $x_1':=x_1+a_1\epsilon$ and $x_2':=x_2+a_2\epsilon$ and
 $s':=x_2'/x_1'+b\epsilon$, we have  $t_0':x_1'=x_2'=0$ and $t_1':s'=x_1'=0$.

Let $t'\in F(\I D,\,\theta)(T)$ represent $(t'_0, t'_1)$.  Set
$z':=\Upsilon_\theta(t')\in \Hilb^d_{F/Y}(T)$.  Let $Z'$ denote the
corresponding subscheme, and $\CI'$ its ideal.  Let's show that $\CI'$
is generated by the following elements:
$$f_r':=(x_1')^{m-r+\delta(r)}(x_2')^{r}
+rb\epsilon (x_1')^{m-r+1+\delta(r)}(x_2')^{r-1}
\text{\quad for } 0\le r\le m.$$
The $f_r'$ reduce to the $f_r$, which generate $\CI$.  Further, $\CI'$
reduces to $\CI$ as $Z'$ is flat over $K[\epsilon]$.  Hence it suffices
to prove that $\CI'$ contains the $f_r'$.

The equation $x_2'/x_1'=s'-b\epsilon$ yields
$$f'_r=(x'_1)^{m+\delta(r)}(s')^r\text{\quad for } 0\le r\le m.$$
  Hence, the pullback of $f'_r$ vanishes
along $e^{(1,\,2)}_TE^{(1)}_T$ to order at least $m$, and along
$E^{(2)}_T$ to order at least $m+2$ since $m+r+\delta(r)\ge m+2$.  Thus
$\CI'$ contains the $f_r'$.

Recall that $T_z{\rm Hilb}^d_{F/Y}(K)=\Hom(\CI,\CO_Z)$. Furthermore, it
follows from the computations above that
$$d_t\Upsilon_\theta(t')(f_r')
 = x_1^{m-r+\delta(r)-1}x_2^{r-1}\left((m-r+\delta(r))x_2a_1
   +rx_1a_2+rx_1^2b\right)$$
for $0\le r\le m$. In particular,  $rx_1^{m-2}x_2^2b\in \CI$ yields
\begin{align*}
d_t\Upsilon_\theta(t')(f_1') &=  mx_1^{m-1}x_2a_1 +x_1^ma_2+x_1^{m+1}b\\
d_t\Upsilon_\theta(t')(f_0') &= (m+2)x_1^{m+1}a_1, \text{\quad and}\\
d_t\Upsilon_\theta(t')(f_3') &=(m-3)x_1^{m-4}x_2^3a_1 +3x_1^{m-3}x_2^2a_2.
\end{align*}
 Recall that,  in $K[x_1,x_2]/\CI$, the monomials 
$$x_1^{m-1},\ x_1^{m-1}x_2,\ x_1^{m+1}, \text{ and }
 x_1^{m-r+\delta(r)}x_2^{l} \text{ for }0\le l<r\le m$$
 are
linearly independent.  But $m\ge3$, so  at least one of the
coefficients $m$, $m+2$, and $m-3$ is prime to the
characteristic. Thus, $\ker(d_t\Upsilon_\theta)=0$, and we are
done.
\end{proof}

\end{document}